\documentclass[a4paper,reqno,10pt]{amsart}

\usepackage{hyperref}
\usepackage{a4wide}

\usepackage{amssymb}
\usepackage{amstext}
\usepackage{amsmath}
\usepackage{amscd}
\usepackage{amsthm}
\usepackage{amsfonts}

\usepackage{graphicx}
\usepackage{latexsym}
\usepackage{mathrsfs}

\usepackage{enumitem}
\setlist[enumerate,1]{label={\upshape(\arabic*)}}
\setlist[enumerate,2]{label={\upshape(\alph*)}}

\usepackage{tikz}
\usetikzlibrary{cd,arrows,matrix,backgrounds,positioning,calc,decorations.markings,decorations.pathmorphing,decorations.pathreplacing}
\tikzset{black/.style={circle,fill=black,inner sep=3pt,outer sep=3pt},
         white/.style={circle,fill=white,draw=black,inner sep=3pt,outer sep=3pt},
}

\usepackage{array}
\newcolumntype{C}{>{$}c<{$}}

\usepackage{makecell}
\usepackage{adjustbox}

\newtheorem{theorem}{Theorem}[section]
\newtheorem{theoremi}{Theorem}
\newtheorem{corollaryi}[theoremi]{Corollary}

\newtheorem{corollary}[theorem]{Corollary}
\newtheorem{lemma}[theorem]{Lemma}
\newtheorem*{lemma*}{Lemma}
\newtheorem*{theorem*}{Theorem}
\newtheorem{proposition}[theorem]{Proposition}
\newtheorem{definition-proposition}[theorem]{Definition-Proposition}

\newtheorem{conjecture}[theorem]{Conjecture}

\theoremstyle{definition}
\newtheorem{definition}[theorem]{Definition}

\newtheorem{remark}[theorem]{Remark}
\newtheorem{example}[theorem]{Example}

\newtheorem*{ack}{Acknowledgement}
\newtheorem*{conv}{Conventions and notation}
\newtheorem*{org}{Organization}

\newcommand{\qedb}{\hfill\blacksquare}

\newcommand{\la}{\langle}
\newcommand{\ra}{\rangle}
\newcommand{\al}{\alpha}
\newcommand{\be}{\beta}

\renewcommand{\AA}{\mathcal{A}}

\newcommand{\CC}{\mathcal{C}}
\newcommand{\MM}{\mathcal{M}}
\newcommand{\NN}{\mathcal{N}}
\newcommand{\GG}{\mathcal{G}}

\newcommand{\DD}{\mathcal{D}}

\newcommand{\FF}{\mathcal{F}}
\newcommand{\FFF}{\mathsf{F}}

\renewcommand{\SS}{\mathcal{S}}

\newcommand{\WW}{\mathcal{W}}
\newcommand{\WWW}{\mathsf{W}}

\newcommand{\Z}{\mathbb{Z}}
\renewcommand{\P}{\mathbb{P}}

\newcommand{\R}{\mathbb{R}}

\renewcommand{\top}{\operatorname{top}\nolimits}

\newcommand{\soc}{\operatorname{soc}\nolimits}

\newcommand{\Ext}{\operatorname{Ext}\nolimits}

\newcommand{\End}{\operatorname{End}\nolimits}

\newcommand{\RHom}{\mathbf{R}\strut\kern-.2em\operatorname{Hom}\nolimits}

\newcommand{\Image}{\operatorname{Im}\nolimits}
\newcommand{\Kernel}{\operatorname{Ker}\nolimits}
\newcommand{\Cokernel}{\operatorname{Coker}\nolimits}

\newcommand{\coker}{\Cokernel}
\newcommand{\im}{\Image}
\renewcommand{\ker}{\Kernel}

\newcommand{\ov}{\overline}

\newcommand{\ot}{\leftarrow}

\DeclareMathOperator{\moduleCategory}{\mathsf{mod}} \renewcommand{\mod}{\moduleCategory}

\DeclareMathOperator{\simp}{\mathsf{sim}}

\DeclareMathOperator{\torf}{\mathsf{torf}}

\DeclareMathOperator{\brick}{\mathsf{brick}}

\DeclareMathOperator{\sbrick}{\mathsf{sbrick}}
\DeclareMathOperator{\mbrick}{\mathsf{mbrick}}

\DeclareMathOperator{\ccmbrick}{\mathsf{mbrick_{c.c.}}}

\DeclareMathOperator{\lSchur}{\mathsf{Schur_L}}

\DeclareMathOperator{\wide}{\mathsf{wide}}

\DeclareMathOperator{\mmax}{\mathsf{max}}

\DeclareMathOperator{\sub}{\mathsf{sub}}

\DeclareMathOperator{\Filt}{\mathsf{Filt}}

\DeclareMathOperator{\add}{\mathsf{add}}

\newcommand{\iso}{\cong}

\newcommand{\defl}{\twoheadrightarrow}

\newcommand{\EE}{\mathcal{E}}

\numberwithin{equation}{section}

\begin{document}
\title[Monobricks]{Monobrick, a uniform approach to torsion-free classes and wide subcategories}

\author[H. Enomoto]{Haruhisa Enomoto}
\address{Graduate School of Science, Osaka Prefecture University, 1-1 Gakuen-cho, Naka-ku, Sakai, Osaka 599-8531, Japan}
\email{the35883@osakafu-u.ac.jp}
\begin{abstract}
  For a length abelian category, we show that all torsion-free classes can be classified by using only the information on bricks, including non functorially-finite ones. The idea is to consider the set of simple objects in a torsion-free class, which has the following property: it is a set of bricks where every non-zero map between them is an injection. We call such a set a monobrick.
  In this paper, we provide a uniform method to study torsion-free classes and wide subcategories via monobricks.
  We show that monobricks are in bijection with left Schur subcategories, which contains all subcategories closed under extensions, kernels and images, thus unifies torsion-free classes and wide subcategories. Then we show that torsion-free classes bijectively correspond to cofinally closed monobricks.
  Using monobricks, we deduce several known results on torsion(-free) classes and wide subcategories (e.g. finiteness result and bijections) in length abelian categories, without using $\tau$-tilting theory.
  For Nakayama algebras, left Schur subcategories are the same as subcategories closed under extensions, kernels and images, and we show that its number is related to the large Schr\"oder number.
\end{abstract}

\maketitle

\tableofcontents

\section{Introduction}
For a finite-dimensional algebra $\Lambda$, several classes of subcategories of $\mod\Lambda$ have been investigated in the representation theory of algebras. Among them, \emph{torsion classes} and \emph{torsion-free classes} have been central, together with their connection to the tilting theory and various triangulated categories.

Recently, Adachi-Iyama-Reiten's paper \cite{AIR} made a major breakthrough in a classification of these subcategories. They show that functorially finite torsion-free classes can be classified using \emph{support $\tau^-$-tilting modules}. Their method is to consider $\Ext$-injective objects.

In this paper, we take a different approach, which enables us to classify \emph{all} torsion-free classes in \emph{any length abelian categories}. Our method is to consider \emph{simple objects} (Definition \ref{def:simple}).
We observe that every simple object in a torsion-free class is a brick (a module with a division endomorphism ring), and we classify torsion-free classes using only the information on bricks:
\begin{theoremi}[= Theorem \ref{thm:torfproj}]\label{thm:A}
  Let $\AA$ be a length abelian category. Then we have a bijection between the following sets.
  \begin{enumerate}
    \item The set of all torsion-free classes $\FF$ in $\AA$.
    \item $\{ \MM \, | \, \text{$\MM$ is a set of bricks in $\AA$ satisfying the following two conditions: }\}$
    \begin{itemize}[font=\upshape]
      \item[(MB)] Every non-zero map between objects in $\MM$ is injective.
      \item[(CC)] If there is an injection $N \hookrightarrow M$ for a brick $N \not\in\MM$ and $M \in \MM$, then there is a non-zero non-injection $N \to M'$ for some $M' \in \MM$.
    \end{itemize}
  \end{enumerate}
  The map from {\upshape (1)} to {\upshape (2)} is given by the set $\simp\FF$ of simple objects in $\FF$, and from {\upshape (2)} to {\upshape (1)} is given by taking the extension closure $\Filt\MM$ of $\MM$.
\end{theoremi}

\begin{example}[= Example \ref{ex:2nak}]
  Let $\Lambda$ be \emph{any} Nakayama algebra whose quiver is given by $1 \rightleftarrows 2$.
  Then there are 4 bricks in $\mod\Lambda$, namely, $1,2, \substack{1\\2}, \substack{2\\1}$.
  Only from the information on (non-)injections between them, we can combinatorially list all sets $\MM$ satisfying (MB) and (CC) above, namely,
  $\varnothing, \{1\}, \{2\}, \{1,\substack{2\\1} \}, \{2,\substack{1\\2}\},\{1,2\}$. Hence there are 6 torsion-free classes in $\mod\Lambda$, namely,
  $0 = \Filt \varnothing, \Filt\{1\}, \Filt\{2\}, \Filt \{1,\substack{2\\1} \}, \Filt\{2,\substack{1\\2}\}, \Filt \{1,2\} = \mod\Lambda$.
\end{example}

In this paper, we call a set $\MM$ of bricks satisfying (MB) a \emph{monobrick}. A well-known example is a \emph{semibrick}, a pairwise Hom-orthogonal set of bricks. It is classical that semibricks in $\AA$ are in bijection with wide subcategories in $\AA$ by the same maps as in Theorem \ref{thm:A} (c.f. \cite[1.2]{ringel}). The aim of this paper is to provide a uniform theory to study monobricks and several kinds of subcategories including torsion-free classes and wide subcategories, thereby giving a systematic framework for studying these subcategories.

Our starting point is the bijection between $\mbrick\AA$, the set of monobricks in $\AA$, and $\lSchur\AA$, the set of \emph{left Schur subcategories}. A left Schur subcategory is a category whose simple objects satisfy the one-sided Schur's lemma (see Definition \ref{def:rs}).
The class of left Schur subcategories contains any subcategories of $\AA$ which are closed under extensions, kernels and images, thus unifies torsion-free classes and wide subcategories.
\begin{theoremi}[= Theorems \ref{thm:main}, \ref{thm:torfproj}, \ref{thm:wideproj}]\label{thm:c}
  Let $\AA$ be a length abelian category. Then we have a bijection between the set of left Schur subcategories and monobricks in $\AA$:
  \[
  \begin{tikzcd}
    \lSchur\AA \rar["\simp", shift left] & \mbrick\AA \lar["\Filt", shift left]
  \end{tikzcd}
  \]
   Moreover, this bijection restricts to the following bijections:
  \begin{itemize}
    \item $\wide\AA \rightleftarrows \sbrick\AA$ between the set of wide subcategories of $\AA$ and semibricks in $\AA$, and
    \item $\torf\AA \rightleftarrows \ccmbrick\AA$ between the set of torsion-free classes in $\AA$ and cofinally closed monobricks in $\AA$ (monobricks satisfying {\upshape (CC)} in Theorem \ref{thm:A}).
  \end{itemize}
\end{theoremi}
In the case of Nakayama algebras, we show that left Schur subcategories are precisely subcategories closed under extensions, kernels and images (Theorem \ref{thm:nakeic}), and the number of left Schur subcategories is related to the large Schr\"oder number (Theorem \ref{thm:nakcount}).

We establish Theorem \ref{thm:c} by using two natural maps $\WWW \colon \lSchur\AA \defl \wide\AA$ and $\FFF \colon \lSchur\AA \defl \torf\AA$, where $\WWW(\EE)$ is the same as in \cite{IT,MS} (Definition \ref{def:wdef}) and $\FFF(\EE)$ is the smallest torsion-free class containing $\EE$.
We describe these maps in terms of a natural \emph{poset structure} of monobricks $\MM$, namely, $L \leq M$ in $\MM$ if there is an injection $L \hookrightarrow M$. Then semibricks and cofinally closed monobricks can be characterized by this poset structure (see Proposition \ref{prop:sbmax} and Definition \ref{def:cof}). Now the maps $\WWW$ and $\FFF$ are easily described in terms of the poset structure as follows.
\begin{theoremi}[= Theorems \ref{thm:torfproj}, \ref{thm:wideproj}]\label{thm:d}
  Let $\AA$ be a length abelian category.
  \begin{enumerate}
    \item The following diagram commutes, and the horizontal maps are bijections.
    \[
    \begin{tikzcd}
      \lSchur\AA \dar["\WWW", twoheadrightarrow] \rar["\simp", shift left] & \mbrick\AA \lar["\Filt", shift left] \dar["\mmax", twoheadrightarrow] \\
      \wide\AA \rar["\simp", shift left] & \sbrick\AA \lar["\Filt", shift left]
    \end{tikzcd}
    \]
    Here $\mmax \MM$ for a monobrick $\MM$ denotes the set of maximal elements of $\MM$.
    \item The following diagram commutes, and the horizontal maps are bijections.
    \[
    \begin{tikzcd}
      \lSchur\AA \dar["\FFF", twoheadrightarrow] \rar["\simp", shift left] & \mbrick\AA \lar["\Filt", shift left] \dar["(\ov{-})", twoheadrightarrow] \\
      \torf\AA \rar["\simp", shift left] & \ccmbrick\AA \lar["\Filt", shift left]
    \end{tikzcd}
    \]
    Here $\ov{\MM}$ for a monobrick $\MM$ denotes the cofinal closure of $\MM$ (Definition \ref{def:closure}).
  \end{enumerate}
\end{theoremi}
As an application, we can quickly prove the finiteness result in \cite{DIJ}: $\torf\AA$ is a finite set if and only if there are only finitely many bricks (Theorem \ref{thm:bfinite}). In addition, we can easily deduce the following bijections between $\torf\AA$ and $\wide\AA$ using only some combinatorial observation on posets. This was proved in \cite{MS} in the case of finite-dimensional algebras by using $\tau$-tilting theory.
\begin{corollaryi}[= Corollary \ref{cor:brickfinbij}]
  Let $\AA$ be a length abelian category. Suppose that $\torf\AA$ is a finite set. Then the maps $\WWW \colon \torf\AA \to \wide\AA$ and $\FFF \colon \wide\AA \to \torf\AA$ are mutually inverse bijections.
\end{corollaryi}

{\bf Comparison to $\tau$-tilting theory.}
For the convenience of the reader, we summarize advantages and disadvantages of monobricks compared to $\tau$-tilting theory.

(Advantages)
\begin{itemize}
  \item $\tau$-tilting theory uses $\Ext$-projectives, while we use simple objects. This enables us to work with any length abelian categories, where there may not be any projective objects.
  \item $\tau$-tilting theory cannot classify non functorially finite cases, while monobricks can. This is because non-functorially finite subcategories may not have $\Ext$-projectives.
  \item Using monobricks, we can study both wide subcategories and torsion-free classes in the same framework, and the relation between them becomes more transparent.
  \item Left Schur subcategories, or subcategories closed under extensions, kernels and images, seem to be new objects to study. Our enumerative result on Nakayama algebras via the Schr\"oder number suggests that there are more hidden combinatorics in other algebras.
\end{itemize}

(Disadvantages)
\begin{itemize}
  \item In general, left Schur subcategories are complicated to deal with. Actually, there are left Schur subcategories which are not even closed under direct summands (Example \ref{ex:213ex}).
  \item One of the benefits of $\tau$-tilting theory is a \emph{mutation}, which provides a way to create various torsion-free classes starting from $\mod\Lambda$. So far, we have no such theory for monobricks.
  \item We cannot investigate functorial finiteness by monobricks. More precisely, two isomorphic monobricks (as posets) can correspond to functorially finite and non-functorially finite torsion-free classes (Example \ref{ex:kro}).
\end{itemize}

Finally, we should mention the relation of this paper to \cite{asai}, where a bijection between functorially finite torsion-free classes and semibricks satisfying some conditions was established. Although we cannot reprove his result (due to the last diasadvantage), his map can be easily described via monobricks: $\FF \mapsto \mmax(\simp\FF)$. See Remark \ref{rem:label} for more details.

\begin{org}
  This paper is organized as follows.
  In Section \ref{sec:2}, we study basic properties of left Schur subcategories and monobricks, and establish a bijection between them.
  In Section \ref{sec:3}, we study the cofinal closure $\ov{\MM}$ and show its relation to torsion-free classes.
  In Section \ref{sec:4}, we study the map $\WWW$ and its relation to $\mmax\MM$.
  In Section \ref{sec:5}, we apply previous results to show results on torsion-free classes and wide subcategories.
  In Section \ref{sec:6}, we give a combinatorial classification of monobricks over Nakayama algebras, and count their number.
  In Section \ref{sec:ex}, we show some examples of the classification of monobricks and the computation of $\ov{\MM}$ and $\mmax \MM$.
\end{org}

\begin{conv}
  Throughout this paper, \emph{we assume that all categories are skeletally small}, that is, the isomorphism classes of objects form a set. In addition, \emph{all subcategories are assumed to be full and closed under isomorphisms}. We often identify an isomorphism class in a category with its representative.
  \emph{We always denote by $\AA$ a skeletally small length abelian category}, that is, an abelian category in which every object has finite length.
  For a collection $\CC$ of objects in $\AA$, we denote by $\add\CC$ the subcategory of $\AA$ consisting of direct summands of finite direct sums of objects in $\CC$.
  For a finite-dimensional algebra $\Lambda$, we denote by $\mod\Lambda$ the category of finitely generated right $\Lambda$-modules.
  For a set $A$, we denote by $\#A$ its cardinality.
\end{conv}

\section{Bijection between monobricks and left Schur subcategories}\label{sec:2}
First we introduce a \emph{monobrick} in a length abelian category $\AA$. Recall that a \emph{brick} in $\AA$ is an object $M$ such that $\End_\AA(M)$ is a division ring.
\begin{definition}
  Let $\MM$ be a set of isomorphism classes of bricks in $\AA$.
  \begin{enumerate}
    \item $\MM$ is called a \emph{monobrick} if every morphism between elements of $\MM$ is either zero or an injection in $\AA$. We denote by $\mbrick \AA$ the set of monobricks in $\AA$.
    \item $\MM$ is called a \emph{semibrick} if every morphism between elements of $\MM$ is either zero or an isomorphism. We denote by $\sbrick \AA$ the set of semibricks in $\AA$.
  \end{enumerate}
\end{definition}
Note that the assumption that a monobrick $\MM$ consists of bricks is automatically satisfied if we require the above property, since every non-zero endomorphism of $M$ in $\MM$ should be an injection, thus an isomorphism since $M$ has finite length.

Clearly every semibrick is a monobrick, thus $\sbrick\AA \subset \mbrick\AA$ holds.
Next we introduce \emph{left Schur subcategories} of $\AA$. Roughly speaking, it is an extension-closed subcategory of $\AA$ such that the ``one-sided Schur's lemma'' holds. Let us define some notations.
\begin{definition}\label{def:simple}
  Let $\EE$ be a subcategory of $\AA$.
  \begin{enumerate}
    \item $\EE$ is \emph{closed under extensions} or \emph{extension-closed} in $\AA$ if it satisfies the following condition:
    for every short exact sequence
    \[
    \begin{tikzcd}
      0 \rar & X \rar & Y \rar & Z \rar & 0
    \end{tikzcd}
    \]
    in $\AA$, if $X$ and $Z$ belong to $\EE$, then so does $Y$.
    \item Suppose that $\EE$ is extension-closed in $\AA$. Then a non-zero object $M$ in $\EE$ is a \emph{simple object in $\EE$} if there is \emph{no} exact sequence of the form
    \[
    \begin{tikzcd}
      0 \rar & L \rar & M \rar & N \rar & 0
    \end{tikzcd}
    \]
    in $\AA$ satisfying $L,M,N \in \EE$ and $L,N \neq 0$. We denote by $\simp\EE$ the set of isomorphism classes of simple objects in $\EE$.
  \end{enumerate}
\end{definition}
Clearly $\simp\AA$ is nothing but the set of the usual simple objects in an abelian category $\AA$. Thus $\simp\EE$ is an analogue of simple objects \emph{inside $\EE$}. An extension-closed subcategory of $\AA$ can naturally be regarded as an exact category, and the notion of simple objects are invariant under an equivalence of exact categories. Thus simple objects be considered as one of the invariants of exact categories. Actually, the validity of the Jordan-H\"older type property in $\EE$ can be characterized using simple objects in \cite{eno:jhp}.

To define and study left Schur subcategories, the following terminology is useful.
\begin{definition}
  Let $\CC$ be a collection of objects in $\AA$. Then a non-zero object $M \in \AA$ is \emph{left Schurian for $\CC$} if every morphism $M \to C$ with $C \in \CC$ is either zero or an injection in $\AA$.
\end{definition}
Note that we do not require that $M$ belongs to $\CC$. It is clear that a collection $\MM$ of non-zero objects in $\AA$ is a monobrick if and only if every object in $\MM$ is left Schurian for $\MM$.
Simple objects in $\AA$ (in the usual sense) are left Schurian for any collection $\CC$.

The fundamental relation between left Schurian objects and simple objects is as follows.
\begin{proposition}\label{prop:rsmb}
  Let $\EE$ be an extension-closed subcategory of $\AA$. Then the following hold.
  \begin{enumerate}
    \item $\{ M \in \EE \, | \, \text{$M$ is left Schurian for $\EE$} \}$ is a monobrick.
    \item If $M$ in $\EE$ is left Schurian for $\EE$, then $M$ is simple in $\EE$.
  \end{enumerate}
\end{proposition}
\begin{proof}
  (1)
  Suppose that $M$ and $N$ in $\EE$ are left Schurian for $\EE$. Then since $M$ is left Schurian for $\EE$, every morphism $M \to N$ is either zero or an injection. Thus the assertion holds.

  (2)
  Suppose that $M$ in $\EE$ is left Schurian for $\EE$, and take an exact sequence
  \[
  \begin{tikzcd}
    0 \rar & L \rar & M \rar["\pi"] & N \rar & 0
  \end{tikzcd}
  \]
  in $\AA$ with $L,N \in \EE$. Then $\pi$ should be either zero or an injection. In the former case, we have $N=0$, and in the latter, we have $L=0$. Thus $M$ is simple in $\EE$.
\end{proof}
Then we can define a left Schur subcategory as follows.
\begin{definition}\label{def:rs}
  A subcategory $\EE$ of $\AA$ is \emph{left Schur} if it satisfies the following conditions:
  \begin{enumerate}
    \item $\EE$ is closed under extensions in $\AA$.
    \item Every simple object in $\EE$ is left Schurian for $\EE$, that is, for a simple object $M$ in $\EE$, every morphism $M \to X$ with $X \in \EE$ is either zero or an injection in $\AA$.
  \end{enumerate}
  We denote by $\lSchur \AA$ the set of left Schur subcategories of $\AA$.
\end{definition}
The following immediately follows from Proposition \ref{prop:rsmb}:
\begin{corollary}\label{cor:rssimp}
  Let $\EE$ be a left Schur subcategory of $\AA$ and $M$ an object in $\EE$. Then $M$ is simple in $\EE$ if and only if $M$ is left Schurian for $\EE$. Moreover, $\simp\EE$ is a monobrick.
\end{corollary}

We are going to show that all wide subcategories and torsion-free classes are left Schur. Let us recall the definitions of these subcategories.
\begin{definition}
  Let $\EE$ be a subcategory of $\AA$.
  \begin{enumerate}
    \item $\EE$ is \emph{closed under kernels (resp. cokernels, images)} if for every morphism $X \to Y$ in $\EE$, we have $\ker f$ (resp. $\coker f$, $\im f$) belongs to $\EE$.
    \item $\EE$ is \emph{closed under subobjects (resp. quotients)} if every subobject (resp. quotient object) of $X$ belongs to $\EE$ for every $X$ in $\EE$.
    \item $\EE$ is a \emph{torsion-free class (resp. torsion class) in $\AA$} if it is closed under extensions and subobjects (resp. extensions and quotients) in $\AA$. We denote by $\torf\AA$ the set of torsion-free classes in $\AA$.
    \item $\EE$ is a \emph{wide subcategory of $\AA$} if it is closed under extensions, kernels and cokernels. We denote by $\wide\AA$ the set of wide subcategories of $\AA$.
  \end{enumerate}
\end{definition}
It can be easily shown that every wide subcategory or every torsion-free class in $\AA$ is closed under extensions, kernels and images. We prove that this condition implies left Schur, thus $\wide\AA \subset \lSchur\AA$ and $\torf\AA \subset \lSchur\AA$ hold.
\begin{proposition}\label{prop:0kerlschur}
  Let $\EE$ be a subcategory of $\AA$ which is closed under extensions, kernels and images in $\AA$. Then $\EE$ is a left Schur subcategory of $\AA$.
\end{proposition}
\begin{proof}
  Let $M$ be a simple object in $\EE$ and $f \colon M \to X$ be a morphism with $X \in \EE$. Then we have the following exact sequence in $\AA$:
  \[
  \begin{tikzcd}
    0 \rar & \ker f \rar & M \rar & \im f \rar & 0
  \end{tikzcd}
  \]
  Since $\EE$ is closed under kernels and images, we have $\ker f, \im f \in \EE$. Thus either $\ker f = 0$ or $\im f = 0$ since $M$ is simple in $\EE$. In the former case, we have that $f$ is an injection in $\AA$, and in the latter, we have $f = 0$. Thus $M$ is left Schurian for $\EE$, hence $\EE$ is left Schur.
\end{proof}

\begin{example}
  Let $k$ be a field and $Q$ be the quiver $1 \ot 2 \ot 3$. Then the Auslander-Reiten quiver of $\mod kQ$ is as follows:
  
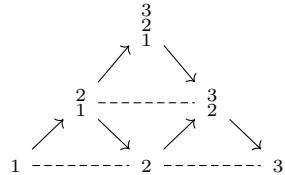
\[
  \begin{tikzcd}[column sep = small, row sep = small]
    & & \substack{3 \\ 2 \\ 1} \ar[rd] \\
    & \substack{2 \\ 1} \ar[rr,-,dashed] \ar[ru] \ar[rd] & & \substack{3 \\ 2} \ar[rd] \\
    \substack{1} \ar[ru] \ar[rr,-,dashed] & & \substack{2} \ar[ru]\ar[rr,-,dashed] & & \substack{3}
  \end{tikzcd}
  \]
  Now $\EE_1 =\add \{ \substack{1},\substack{2 \\ 1}, \substack{3 \\ 2 \\ 1}, \substack{2}\}$ is a torsion-free class, and $\simp\EE_1 = \{ \substack{1}, \substack{2}, \substack{3 \\2 \\ 1} \}$.
  It can be checked that every simple object in $\EE_1$ is left Schurian for $\EE_1$, thus it is a left Schur subcategory (this follows also from Proposition \ref{prop:0kerlschur}).
  On the other hand, consider $\EE_2 = \add \{\substack{2 \\ 1}, \substack{3 \\ 2 \\ 1}, \substack{2} \}$. This subcategory is closed under extensions, and all the three indecomposables are simple objects in $\EE_2$. However, we have a non-zero non-injection $\substack{2 \\ 1} \defl \substack{2}$. Thus $\EE_2$ is not a left Schur subcategory.
\end{example}

For a left Schur subcategory $\EE$ of $\AA$, we have a monobrick $\simp\EE$ by Corollary \ref{cor:rssimp}. Conversely, for a given monobrick $\MM$, we will construct a left Schur subcategory whose simples are $\MM$. To do this, we will use the following operation.
\begin{definition}
  Let $\CC$ be a collection of objects in $\AA$.
  Then $\Filt\CC$ denotes the subcategory of $\AA$ consisting of objects $X$ such that there is a chain
  \[
  0 = X_0 < X_1 < \cdots < X_n = X
  \]
  of subobjects of $X$ such that $X_i/X_{i-1}$ is in $\CC$ for each $i$. We call such a chain a \emph{$\CC$-filtration} of $X$, and $n$ the \emph{length} of this $\CC$-filtration.
\end{definition}
It follows from the Noether isomorphism theorem that $\Filt\CC$ is extension-closed, and it is obvious from the construction that it is the smallest extension-closed subcategory of $\AA$ containing $\CC$.

Now we can state our first main result of this paper.
\begin{theorem}\label{thm:main}
  Let $\AA$ be a length abelian category. Then $\simp$ and $\Filt$ give mutually inverse bijections between left Schur subcategories of $\AA$ and monobricks in $\AA$, which extends the bijection between wide subcategories and semibricks:
  \[
  \begin{tikzcd}[row sep = small]
   \lSchur\AA \rar["\simp", shift left] & \mbrick\AA \lar["\Filt", shift left] \\
    \wide\AA \rar["\simp", shift left] \uar[hookrightarrow] & \sbrick\AA \lar["\Filt", shift left] \uar[hookrightarrow]
  \end{tikzcd}
  \]
\end{theorem}

We need some preparation to prove it. For two collections $\CC$ and $\DD$ of objects in $\AA$, we denote by $\CC * \DD$ the subcategory of $\AA$ consisting of objects $X$ such that there is an exact sequence
\[
\begin{tikzcd}
  0 \rar & C \rar & X \rar & D \rar & 0
\end{tikzcd}
\]
in $\AA$ with $C\in \CC$ and $D \in \DD$. As for this, the following lemma in \cite[Lemma 3.10]{eno:binv} is useful. We give a proof for the convenience of the reader.
\begin{lemma}\label{lem:rsext}
  Let $M$ be an object in $\AA$. If $M$ is left Schurian for two collections $\CC$ and $\DD$ of objects in $\AA$, then $M$ is left Schurian also for $\CC * \DD$.
\end{lemma}
\begin{proof}
  Take a short exact sequence in $\AA$
  \[
  \begin{tikzcd}
    0 \rar & C \rar["\iota"] & X \rar["\pi"] & D \rar & 0
  \end{tikzcd}
  \]
  with $C \in \CC$ and $D \in \DD$.
  Let $\varphi \colon M \to X$ be any morphism, and we will prove that $\varphi$ is either zero or an injection in $\AA$.
  Consider the following commutative diagram.
  \[
  \begin{tikzcd}
    & & M \dar["\varphi"] \\
    0 \rar & C \rar["\iota"] & X \rar["\pi"] & D \rar & 0
  \end{tikzcd}
  \]
  Since $M$ is left Schurian for $\DD$, either $\pi\varphi$ is an injection or $\pi\varphi = 0$. In the former case, $\varphi$ is an injection, so suppose the latter. Then there exists a morphism $\ov{\varphi}\colon M \to C$ with $\varphi = \iota\ov{\varphi}$.
  Since $M$ is left Schurian for $\CC$, we have that $\ov{\varphi}$ is either zero or injective. Thus $\varphi$ is either zero or injective respectively.
\end{proof}
Now we are ready to prove Theorem \ref{thm:main}.
\begin{proof}[Proof of Theorem \ref{thm:main}]
  For a left Schur subcategory $\EE$ of $\AA$, we have $\simp\EE \in \mbrick\AA$ by Corollary \ref{cor:rssimp}, thus we have a map $\simp \colon \lSchur\AA \to \mbrick\AA$.

  For the converse direction, let $\MM$ be a monobrick in $\AA$. Then $\Filt\MM$ is closed under extensions in $\AA$. We will prove that $\Filt\MM$ is a left Schur subcategory of $\AA$.
  We show the following claim:

  {\bf (Claim)}:
  Let $\MM$ be a monobrick. Then the following are equivalent for $M\in \Filt\MM$.
  \begin{enumerate}
    \item $M$ is simple in $\Filt\MM$.
    \item $M$ is in $\MM$.
    \item $M$ is left Schurian for $\Filt\MM$.
  \end{enumerate}

  \noindent
  \emph{Proof of (Claim).}

  (1) $\Rightarrow$ (2):
  By the construction of $\Filt\MM$ and the definition of a simple object, every simple object in $\Filt\MM$ should belong to $\MM$.

  (2) $\Rightarrow$ (3):
   Let $M \in \MM$. Since $\MM$ is a monobrick, $M$ is left Schurian for $\MM$. Then by using Lemma \ref{lem:rsext} repeatedly, $M$ is left Schurian for $\Filt\MM$.

  (3) $\Rightarrow$ (1):
  This follows from Proposition \ref{prop:rsmb}. $\qedb$

  In particular, the implication (1) $\Rightarrow$ (3) implies that $\Filt\MM$ is left Schur.
  Therefore we obtain a map $\Filt\colon \mbrick\AA \to \lSchur\AA$. We will prove that these maps are mutually inverse to each other.
  Since $\AA$ is a length category, it can be easily shown by induction on lengths that $\EE = \Filt(\simp\EE)$ holds for any extension-closed subcategory $\EE$ of $\AA$. Thus $\Filt \circ \simp \colon \lSchur\AA \to \lSchur\AA$ is the identity.
  Conversely, (Claim) implies $\simp(\Filt\MM) = \MM$ for a monobrick $\MM$. Therefore, we have mutually inverse bijections $\simp\colon \lSchur\AA \to \mbrick\AA$ and $\Filt \colon \mbrick\AA \to \lSchur\AA$.

  Finally, we claim that these bijections restrict to bijections between wide subcategories and semibricks.
  Ringel's result \cite[1.2]{ringel} implies that $\Filt\MM$ is a wide subcategory of $\AA$ if $\MM$ is a semibrick. Conversely, let $\EE$ be a wide subcategory of $\AA$. Then $\EE$ is an abelian category, and it is easy to see that simple objects in $\EE$ coincide with usual simple objects in an abelian category $\EE$. Thus Schur's lemma in $\EE$ implies that $\simp\EE$ is a semibrick.
\end{proof}

\section{Maps to torsion-free classes and cofinally closed monobricks}\label{sec:3}
In this section, we will show that a left Schur subcategory $\EE$ of $\AA$ is a torsion-free class if and only if $\simp\EE$ is a \emph{cofinally closed monobrick}. Then we construct a map from $\mbrick\AA$ to the set of cofinally closed monobricks, taking the \emph{cofinal closure}, which corresponds to the map $\FFF(\EE)$ which sends $\EE$ to the smallest torsion-free class containing $\EE$.

\subsection{Cofinal extension and cofinal closure of monobricks}
First we observe that each monobrick has a natural poset structure, which will play a central role in this paper.
\begin{definition}\label{def:suborder}
  Let $\MM$ be a monobrick in $\AA$. For $M,N$ in $\MM$, we define $M \leq N$ if there is an injection $M \to N$ in $\AA$. Since $\AA$ is a length category, it is easily checked that $\leq$ is actually a partial order on $\MM$. We call this order the \emph{submodule order on $\MM$}.
\end{definition}

We introduce a notion of \emph{cofinal extension} between monobricks, and \emph{cofinally closed} monobricks.
\begin{definition}\label{def:cof}
  Let $\MM$ and $\NN$ be two monobricks in $\AA$. Then we say that $\NN$ is a \emph{cofinal extension of $\MM$}, or $\MM$ is \emph{cofinal in $\NN$}, if the following conditions are satisfied:
  \begin{enumerate}
    \item $\MM \subset \NN$ holds.
    \item For every $N \in \NN$, there exists $M \in \MM$ satisfying $N \leq M$ in $\NN$.
  \end{enumerate}
  We say that a monobrick $\MM$ is \emph{cofinally closed} if there is no proper cofinal extension of $\MM$. We denote by $\ccmbrick\AA$ the set of cofinally closed monobricks in $\AA$.
\end{definition}
Note that this is a purely poset theoretical notion, and has nothing to do with the actual module structure of each brick.

Cofinal extensions of $\MM$ are closed under unions in the following sense:
\begin{proposition}\label{prop:cofunion}
  Let $\MM$ be a monobrick in $\AA$. Suppose that $\{ \NN_i \, | \, i \in I\}$ is a family of cofinal extensions of $\MM$. Then $\bigcup_{i \in I} \NN_i$ is a cofinal extension of $\MM$ (in particular, it is a monobrick).
\end{proposition}
\begin{proof}
  Clearly, we only have to see that $\NN:= \bigcup_{i \in I} \NN_i$ is actually a monobrick.
  Take $N_1$ and $N_2$ in $\NN$ with $N_1 \in \NN_{i_1}$ and $N_2 \in \NN_{i_2}$, and let $f \colon N_1 \to N_2$ be any map.
  Since $\NN_{i_2}$ is a cofinal extension of $\MM$, there is an injection $\iota \colon N_2 \hookrightarrow M$ with $M \in \MM$. Then the composition $\iota f \colon N_1 \to M$ is a map between elements in $\NN_{i_1}$, thus it should be either zero or an injection since $\NN_{i_1}$ is a monobrick.
  Then the injectivity of $\iota$ implies that $f$ is either zero or an injection.
\end{proof}
This immediately implies the existence of the \emph{largest cofinal extension} of a given monobrick, which is cofinally closed.
\begin{corollary}\label{cor:closure}
  Let $\MM$ be a monobrick in $\AA$. Then the union $\ov{\MM}$ of all cofinal extensions of $\MM$ satisfies the following properties:
  \begin{enumerate}
    \item $\ov{\MM}$ is a cofinal extension of $\MM$.
    \item For every cofinal extension $\NN$ of $\MM$, we have $\MM \subset \NN \subset \ov{\MM}$.
    \item $\ov{\MM}$ is cofinally closed. Moreover, if $\NN$ is a cofinal extension of $\MM$ which is cofinally closed, then $\NN = \ov{\MM}$ holds.
  \end{enumerate}
\end{corollary}
\begin{proof}
  (1), (2)
  Clear from Proposition \ref{prop:cofunion} and the definition of $\ov{\MM}$.

  (3) Let $\MM'$ be a cofinal extension of $\ov{\MM}$. Then it is easy to see that $\MM'$ is also a cofinal extension of $\MM$. Thus (2) implies $\MM' \subset \ov{\MM}$, thus $\MM' = \ov{\MM}$. Therefore, $\ov{\MM}$ is cofinally closed.
  On the other hand, let $\NN$ be a cofinal extension of $\MM$ which is cofinally closed. Then we have $\MM \subset \NN \subset \ov{\MM}$ holds by (2). It is obvious that $\ov{\MM}$ is a cofinal extension of $\NN$, thus we have $\NN = \ov{\MM}$ by the definition of the cofinal closedness.
\end{proof}
\begin{definition}\label{def:closure}
  Let $\MM$ be a monobrick. We denote by $\ov{\MM}$ the union of all cofinal extensions of $\MM$, and call it the \emph{cofinal closure of $\MM$}. Then $\ov{\MM}$ is the unique cofinal extension of $\MM$ which is cofinally closed by Corollary \ref{cor:closure}.
\end{definition}
Taking the cofinal closure defines a map $(\ov{-}) \colon \mbrick \AA \defl \ccmbrick\AA$, which is the identity on $\ccmbrick\AA$ by the definition of cofinal closedness. Similarly, we can check that a monobrick $\MM$ is cofinally closed if and only if $\ov{\MM} = \MM$ holds.

Next, we will characterize the cofinal closure as in the theory of integral extensions of commutative rings. For this purpose, we will introduce the following notion.
\begin{definition}
  Let $\MM$ be a monobrick in $\AA$. We say that a brick $N$ in $\AA$ is \emph{cofinal over $\MM$} if it satisfies the following conditions:
  \begin{enumerate}
    \item There exist $M \in \MM$ and an injection $N \hookrightarrow M$ in $\AA$.
    \item Every map $N \to M'$ with $M' \in\MM$ is either zero or an injection.
  \end{enumerate}
\end{definition}

\begin{proposition}\label{prop:cof1elem}
  Let $\MM$ be a monobrick in $\AA$ and $N$ a brick in $\AA$. Then $N$ is cofinal over $\MM$ if and only if $\MM \cup \{ N \}$ is a cofinal extension of $\MM$.
\end{proposition}
\begin{proof}
  The ``if" part is clear. Conversely, suppose that $N$ is cofinal over $\MM$, and we claim that $\MM \cup \{N\}$ is a cofinal extension of $\MM$. Obviously it suffices to show that $\MM \cup \{ N \}$ is a monobrick.

  Clearly we only have to show that every map $f \colon M \to N$ with $M \in \MM$ is either zero or an injection. By the assumption, there is an injection $\iota \colon N \hookrightarrow M'$ with $M' \in \MM$. Then the composition $\iota f \colon M \to M'$ is a map between elements in $\MM$, thus is either zero or an injection. Since $\iota$ is injective, this implies that $f$ is either zero or an injection.
\end{proof}

Now we can describe a cofinal extension of a monobrick in terms of elements:
\begin{corollary}
  Let $\MM$ be a monobrick in $\AA$ and $\NN$ a set of bricks in $\AA$ satisfying $\MM \subset \NN$ (we do not require that $\NN$ is a monobrick). Then the following are equivalent:
  \begin{enumerate}
    \item $\NN$ is a cofinal extension of $\MM$ (in particular, $\NN$ is a monobrick).
    \item Every element in $\NN$ is cofinal over $\MM$.
  \end{enumerate}
\end{corollary}
\begin{proof}
  (1) $\Rightarrow$ (2):
  It follows immediately from the definition.

  (2) $\Rightarrow$ (1):
  By Proposition \ref{prop:cof1elem}, we have that $\MM \cup \{ N \}$ is a cofinal extension of $\MM$ for each $N \in \NN$.
  Since we have $\NN = \bigcup_{N \in \NN} ( \MM \cup \{ N \})$, Proposition \ref{prop:cofunion} implies that $\NN$ is a cofinal extension of $\MM$.
\end{proof}
Similarly, we have the following description of the cofinal closure.
\begin{corollary}\label{cor:closuredesc}
  Let $\MM$ be a monobrick in $\AA$. Then we have
  \[
  \ov{\MM} = \{ N \in \brick\AA \, | \, \text{ $N$ is cofinal over $\MM$} \}.
  \]
  In particular, $\MM$ is cofinally closed if and only if the following condition is satisfied:
  \begin{itemize}
    \item[\upshape (CC)] If a brick $N$ has an injection $N \hookrightarrow M$ for some $M \in \MM$ and $N \not\in\MM$, then there is some non-zero non-injection $N \to M'$ with $M'\in \MM$.
  \end{itemize}
\end{corollary}
\begin{proof}
  Since $\ov{\MM}$ is a cofinal extension of $\MM$, every element in $\ov{\MM}$ is cofinal over $\MM$.
  Conversely, suppose that a brick $N$ is cofinal over $\MM$. Then $\MM \cup \{ N \}$ is a cofinal extension of $\MM$ by Proposition \ref{prop:cof1elem}. Thus we have $\MM \cup \{ N \} \subset \ov{\MM}$ by Corollary \ref{cor:closure}, hence $N \in \ov{\MM}$.
\end{proof}

\subsection{Torsion-free classes and cofinally closed monobricks}
In this subsection, we will show that the map $\FFF\colon \lSchur\AA \defl\torf\AA$ corresponds to the map $(\ov{-})\colon \mbrick\AA \defl\ccmbrick\AA$ defined in the previous subsection.

First of all, we can construct a torsion-free class from any collection of objects in $\AA$ as follows.
\begin{definition}
  Let $\CC$ be a collection of objects in $\AA$.
  \begin{itemize}
    \item We denote by $\sub\CC$ the collection of all subobjects of objects in $\CC$ (where subobjects mean the usual subobjects in an abelian category $\AA$).
    \item We denote by $\FFF(\CC):= \Filt(\sub\CC)$.
  \end{itemize}
\end{definition}
\begin{lemma}
  Let $\CC$ be a collection of objects in $\AA$. Then $\FFF(\CC)$ is the smallest torsion-free class containing $\CC$.
\end{lemma}
\begin{proof}
  Although this is well-known (e.g. \cite[Lemma 3.1]{MS}), we give a proof here for the convenience of the reader. Clearly it suffices to show that $\FFF(\CC)$ is a torsion-free class in $\AA$. Since $\FFF(\CC) = \Filt(\sub\CC)$ is extension-closed in $\AA$, it is enough to show that $\FFF(\CC)$ is closed under subobjects.

  Take any $M \in \FFF(\CC)$ and its subobject $X \hookrightarrow M$. We will show $X \in \FFF(\CC)$ by induction on the length $n$ of a $(\sub\CC)$-filtration of $M$.
  If $n = 1$, then $M \in \sub\CC$ holds, thus $M$ is a subobject of some $C \in \CC$. Then it follows that $X$ is also a subobject of $C$, which proves $X \in \sub\CC \subset \FFF(\CC)$.

  Now suppose $n > 1$. Then there is a short exact sequence $0 \to L \xrightarrow{\iota} M \xrightarrow{\pi} N \to 0$ with $L,N \in \FFF(\CC)$ such that $L$ and $N$ have $(\sub\CC)$-filtrations of length smaller than $n$. We can obtain the following exact commutative diagram,
  \[
  \begin{tikzcd}
    0 \rar & L \cap X \rar \dar[hookrightarrow] & X \rar \dar[hookrightarrow] & \pi(X) \rar \dar[hookrightarrow]& 0 \\
    0 \rar & L \rar["\iota"] & M \rar["\pi"] & N \rar & 0
  \end{tikzcd}
  \]
  where all the vertical maps are injections. By the induction hypothesis, we have $L \cap X, \pi(X) \in \FFF(\CC)$. Thus $X \in \FFF(\CC)$ holds since $\FFF(\CC)$ is extension-closed.
\end{proof}
The following basic observation is used later.
\begin{lemma}\label{lem:fc}
  Let $\CC$ be a collection of objects in $\AA$. Then $\FFF(\CC) = \FFF(\Filt\CC)$ holds.
\end{lemma}
\begin{proof}
  Since $\CC \subset \Filt\CC \subset \FFF(\Filt\CC)$, we have $\FFF(\CC) \subset \FFF(\Filt\CC)$ by the minimality of $\FFF(\CC)$. On the other hand, $\FFF(\CC) = \Filt(\sub\CC) \supset \Filt\CC$ holds, thus $\FFF(\CC) \supset \FFF(\Filt\CC)$.
\end{proof}
Thus we have the following commutative diagram.
\[
\begin{tikzcd}
  \mbrick\AA \dar[dd,dashed]\rar["\Filt"', "\sim"] \ar[dr, "\FFF"'] & \lSchur\AA \dar["\FFF", twoheadrightarrow] \\
  & \torf\AA \dar[hookrightarrow] \\
  \mbrick\AA & \lSchur\AA\lar["\simp", "\sim"']
\end{tikzcd}
\]

The following claims that the dotted map is nothing but taking the cofinal closure.
\begin{proposition}\label{prop:ccmap}
  Let $\MM$ be a monobrick in $\AA$. Then we have $\ov{\MM} = \simp\FFF(\MM)$.
\end{proposition}
\begin{proof}
  First, we will prove $\MM \subset \simp\FFF(\MM)$, which is equivalent to that every object in $\MM$ is left Schurian for $\FFF(\MM)$ by Corollary \ref{cor:rssimp}.
  Take any $M \in \MM$. Then $M$ is left Schurian for $\MM$ since $\MM$ is a monobrick. Since every object in $\sub\MM$ admits an injection into some object in $\MM$, it is easily checked that $M$ is left Schurian also for $\sub\MM$. Then Lemma \ref{lem:rsext} implies that $M$ is left Schurian for $\Filt(\sub\MM) = \FFF(\MM)$.

  Next, we will prove that $\MM \subset \simp\FFF(\MM)$ is a cofinal extension. Let $X$ be a simple object in $\FFF(\MM) =\Filt(\sub\MM)$. Then clearly we must have $X \in \sub\MM$. It follows that there is an injection $X \hookrightarrow M$ with $M \in \MM$. This shows that $\simp\FFF(\MM)$ is a cofinal extension of $\MM$.

  Therefore, we have $\simp\FFF(\MM) \subset \ov{\MM}$ by Corollary \ref{cor:closure}.
  On the other hand, since $\ov{\MM}$ is a monobrick, every object in $\ov{\MM}$ is left Schurian for $\ov{\MM}$, thus so it is for $\MM$. As $\ov{\MM} \subset \FFF(\MM)$, the same argument as in the first part implies $\ov{\MM} \subset \simp\FFF(\MM)$. Hence $\simp\FFF(\MM) = \ov{\MM}$ holds.
\end{proof}

As a corollary, we have the following description of simple objects in $\FFF(\EE)$ for a left Schur subcategory $\EE$.
\begin{corollary}
  Let $\EE$ be a left Schur subcategory of $\AA$. Then $\simp\FFF(\EE)$ consists of bricks $N$ in $\AA$ which satisfy the following conditions:
  \begin{enumerate}
    \item There is an injection $N \hookrightarrow M$ with $M \in \simp\EE$.
    \item Every map $N \to M'$ with $M' \in \simp\EE$ is either zero or an injection.
  \end{enumerate}
\end{corollary}
\begin{proof}
  This immediately follows from Corollary \ref{cor:closuredesc}, since we have $\simp\FFF(\EE) = \ov{\simp\EE}$ by Proposition \ref{prop:ccmap} and Theorem \ref{thm:main}.
\end{proof}

Now we can state our characterization of torsion-free classes via monobricks.
\begin{theorem}\label{thm:torfproj}
  Let $\AA$ be a length abelian category. Then we have the following commutative diagram, and all the horizontal maps are bijective.
  \[
  \begin{tikzcd}
    \torf\AA \rar["\simp"] \ar[dd, bend right=75, "1"'] \dar[hookrightarrow] & \ccmbrick\AA \lar["\Filt", shift left]\dar[hookrightarrow] \ar[dd, bend left=75, "1"]\\
    \lSchur\AA \rar["\simp", shift left] \dar["\FFF", twoheadrightarrow] & \mbrick\AA \dar["(\ov{-})", twoheadrightarrow] \lar["\Filt", shift left]  \\
     \torf\AA \rar["\simp"] & \ccmbrick\AA \lar["\Filt", shift left]
  \end{tikzcd}
  \]
\end{theorem}
\begin{proof}
  Clearly it suffices to show that the following are equivalent for a monobrick $\MM$ in $\AA$:
  \begin{enumerate}
    \item $\MM$ is cofinally closed.
    \item $\Filt\MM$ is a torsion-free class.
  \end{enumerate}

  (1) $\Rightarrow$ (2):
  If $\MM$ is cofinally closed, then $\ov{\MM} = \MM$ holds. Therefore, we have $\Filt\MM = \Filt\ov{\MM} = \Filt (\simp \FFF(\MM)) = \FFF(\MM)$ by Proposition \ref{prop:ccmap}, thus $\Filt\MM$ is a torsion-free class in $\AA$.

  (2) $\Rightarrow$ (1):
  Since $\Filt\MM$ is a torsion-free class, we have $\FFF(\MM) = \FFF(\Filt\MM) = \Filt\MM$ holds by Lemma \ref{lem:fc}. Thus we have $\ov{\MM} = \simp\FFF(\MM) = \simp (\Filt\MM) = \MM$ by Proposition \ref{prop:ccmap}. Therefore $\MM$ is cofinally closed.
\end{proof}

We can obtain the following characterization of left Schur subcategories:
\begin{corollary}\label{cor:lschurchar}
  Let $\EE$ be a subcategory of $\AA$. Then $\EE$ is left Schur if and only if there exist a torsion-free class $\FF$ of $\AA$ and a subset $\MM$ of $\simp\FF$ such that $\EE = \Filt\MM$ holds.
\end{corollary}
\begin{proof}
  Since $\simp\FF$ (and its subsets) is a monobrick for a torsion-free class $\FF$, the ``if'' part is clear.
  Conversely, let $\EE$ be a left Schur subcategory of $\AA$. Then we have $\simp\EE \subset \ov{\simp\EE} = \simp \FFF(\EE)$ by Proposition \ref{prop:ccmap}. Thus $\FF:= \FFF(\EE)$ satisfies the desired condition.
\end{proof}

As a similar result, we can prove the following.
\begin{corollary}\label{cor:mbrickchar}
  Let $\MM$ be a set of isomorphism classes of bricks in $\AA$. Then $\MM$ is a monobrick if and only if there exists a cofinally closed monobrick $\NN$ of $\AA$ such that $\MM$ is a subset of $\NN$.
\end{corollary}
\begin{proof}
  The ``if'' part is clear, and take $\NN:= \ov{\MM}$ for the ``only if" part.
\end{proof}

\section{Maps to semibricks and wide subcategories}\label{sec:4}
We construct maps $\mmax\colon\mbrick\AA \to \sbrick\AA$ and $\WWW\colon \lSchur\AA \to \wide\AA$, which are the identities if restricted to $\sbrick\AA$ and $\wide\AA$ respectively.
These maps correspond to each other under Theorem \ref{thm:main}.

First we consider the following operation on monobricks.
\begin{proposition}\label{prop:maxsb}
  Let $\MM$ be a monobrick in $\AA$. Then the set $\mmax \MM$ of maximal elements for the submodule order on $\MM$ is a semibrick.
\end{proposition}
\begin{proof}
  It suffices to show that every map $f \colon M \to N$ with $M,N \in \mmax\MM$ is zero if $M \neq N$. Suppose that $f$ is non-zero.
  Then $f$ should be an injection in $\AA$ since $\MM$ is a monobrick. It follows that $M < N$, which contradicts the maximality of $M$.
\end{proof}
Thus we obtain the map $\mmax\colon \mbrick\AA \to \sbrick\AA$.
Actually we have the following characterization of a semibrick in terms of the poset structure.
\begin{proposition}\label{prop:sbmax}
  Let $\SS$ be a monobrick in $\AA$. Then the following are equivalent:
  \begin{enumerate}
    \item $\SS$ is a semibrick.
    \item $\SS$ is a discrete poset, that is, $M \leq N$ in $\SS$ implies $M=N$.
    \item $\mmax \SS = \SS$ holds.
  \end{enumerate}
\end{proposition}
\begin{proof}
  This is immediate from the definitions and Proposition \ref{prop:maxsb}.
\end{proof}

Next we introduce a map $\WWW \colon \lSchur\AA \to \wide\AA$.
This extends the map $\WWW \colon \torf\AA \to \wide \AA$ defined by Marks-\v{S}t\!'ov\'{i}\v{c}ek \cite{MS}.
\begin{definition}\label{def:wdef}
  Let $\EE$ be a left Schur subcategory of $\AA$. Then $\WWW(\EE)$ is the subcategory of $\EE$ consisting of objects $W\in\EE$ satisfying the following condition:
  For every map $f \colon W \to X$ with $X \in \EE$, we have $\coker f \in \EE$, where $\coker f$ denotes the cokernel of $f$ in $\AA$.
\end{definition}

The following is a key lemma to show that $\WWW(\EE)$ is actually a wide subcategory.
\begin{lemma}\label{lem:we}
  Let $\EE$ be a left Schur subcategory of $\AA$. Then the following holds.
  \begin{enumerate}
    \item For $M \in \simp\EE$, the following are equivalent:
    \begin{enumerate}
      \item $M$ belongs to $\WWW(\EE)$.
      \item $M$ is maximal in the submodule order on $\simp\EE$.
      \item Every non-zero morphism $f \colon M \to X$ with $X \in \EE$ is an injection in $\AA$ and satisfies $\coker f \in \EE$.
    \end{enumerate}
    \item If we have a short exact sequence
    \[
    \begin{tikzcd}
      0 \rar & L \rar["i"] & M \rar["p"] & N \rar & 0
    \end{tikzcd}
    \]
    in $\AA$ with $L,M,N \in \EE$, then $M$ is in $\WWW(\EE)$ if and only if both $L$ and $N$ are in $\WWW(\EE)$.
  \end{enumerate}
\end{lemma}
\begin{proof}
  (1)
  Let $M$ be a simple object in $\EE$.

  (a) $\Rightarrow$ (b):
  Suppose that $M$ is not maximal in $\simp\EE$. Then we have a proper injection $\iota \colon M \hookrightarrow M'$ with $M' \in \simp\EE$. Consider the following exact sequence in $\AA$:
  \[
  \begin{tikzcd}
    0 \rar & M \rar["\iota"] & M' \rar & \coker\iota \rar & 0
  \end{tikzcd}
  \]
  Since $\iota$ is not an isomorphism, $\coker\iota$ is non-zero. Then $\coker\iota$ does not belong to $\EE$, since otherwise $M'$ would not be simple in $\EE$. This implies that $M$ does not belong to $\WWW(\EE)$.

  (b) $\Rightarrow$ (c):
  Let $f \colon M \to X$ be a non-zero map with $X \in \EE$. Then $f$ is an injection in $\AA$ since $\EE$ is left Schur. We will prove $\coker f \in \EE$ by induction on the length of a $(\simp\EE)$-filtration of $X$.

  If $X$ belongs to $\simp\EE$, then the maximality of $M$ clearly implies that $f$ should be an isomorphism. Suppose that $X$ has a $(\simp\EE)$-filtration of length $n > 1$.
  Then we have a short exact sequence $0 \to L \xrightarrow{\iota} X \xrightarrow{\pi} N \to 0$ with $L \in \simp\EE$ and $N$ has a $(\simp\EE)$-filtration of length $n-1$ (in particular, $L,N \in \EE$).
  Consider the following diagram:
  \[
  \begin{tikzcd}
    & & M \dar["f"] \\
    0 \rar & L \rar["\iota"] & X \rar["\pi"] & N \rar & 0
  \end{tikzcd}
  \]
  We consider two cases.

  \textbf{(Case 1)}: $\pi f = 0$.
  In this case, there is a map $\ov{f} \colon M \to L$ satisfying $f = \iota \ov{f}$.
  As $f$ is injective, so is $\ov{f}$. Moreover, since $L \in \simp\EE$, the maximality of $M$ in $\simp\EE$ implies that $\ov{f}$ is an isomorphism. Thus $\coker f \iso \coker \iota \iso N \in \EE$ holds.

  {\bf (Case 2)}: $\pi f \neq 0$. In this case, $\pi f$ is an injection with $\coker(\pi f) \in \EE$ by the induction hypothesis. Then we obtain the following exact commutative diagram.
  \[
  \begin{tikzcd}
    & & 0 \dar & 0 \dar \\
    & & M \rar[equal] \dar["f"]& M \dar["\pi f"] \\
    0 \rar & L \rar["\iota"]\dar[equal] & X \dar\rar["\pi"] & N \dar\rar & 0 \\
    0 \rar & L \rar & \coker f \rar \dar& \coker (\pi f) \rar \dar & 0 \\
    & & 0 & 0
  \end{tikzcd}
  \]
  Since $L$ and $\coker(\pi f)$ belong to $\EE$, so does $\coker f$ since $\EE$ is extension-closed.

  (c) $\Rightarrow$ (a):
  Clear from the definition of $\WWW(\EE)$.

  (2)
  Suppose that $L$ and $N$ belong to $\WWW(\EE)$, and we will prove $M \in \WWW(\EE)$. Take any map $f \colon M \to X$ with $X \in \EE$. Then we obtain the following exact commutative diagram, where $\ov{f}$ is a map induced from the universality of the cokernel $N$ of $i$.
  \[
  \begin{tikzcd}
    0 \rar & L \dar[equal]\rar["i"] & M \dar["f"] \rar["p"] & N \rar \dar["\ov{f}"] & 0 \\
    & L \rar["fi"] & X \rar & \coker (fi) \rar & 0
  \end{tikzcd}
  \]
  Since $L$ is in $\WWW(\EE)$ and $X$ is in $\EE$, we have $\coker(fi) \in \EE$. Therefore, we have $\coker \ov{f} \in \EE$ since $N$ is in $\WWW(\EE)$. On the other hand, it can be shown that the right square is a pushout diagram. Thus $\coker f \iso \coker \ov{f}$ holds, which proves $\coker f \in \EE$. Therefore $M \in \WWW(\EE)$ holds.

  Conversely, suppose that $M$ belongs to $\WWW(\EE)$, and we will show that $L$ and $N$ belong to $\WWW(\EE)$.
  First we will prove $L \in \WWW(\EE)$. Take any map $f \colon L \to X$ with $X \in \EE$. Then by taking the pushout, we obtain the following exact commutative diagram.
  \[
  \begin{tikzcd}
    0 \rar & L \dar["f"]\rar["i"] & M \dar["\ov{f}"] \rar["p"] & N \rar \dar[equal] & 0 \\
    0 \rar & X \rar & E \rar & N \rar & 0
  \end{tikzcd}
  \]
  Since the left square is pushout, we have $\coker f \iso \coker \ov{f}$. On the other hand, we have $E \in \EE$ since $\EE$ is extension-closed and $X,N \in \EE$. Thus $\coker\ov{f} \in \EE$ holds by $M \in \WWW(\EE)$. Therefore $\coker f\in\EE$, which proves $L \in \WWW(\EE)$.

  Next we will prove $N \in \WWW(\EE)$. Take any map $f \colon N \to X$ with $X \in \EE$. Then since $p$ is a surjection, $\coker f \iso \coker (fp)$ holds. On the other hand, $\coker (fp) \in \EE$ holds by $M \in \WWW(\EE)$ and $X \in \EE$. Therefore $\coker f \in \EE$, which proves $N \in \WWW(\EE)$.
\end{proof}

Now we are ready to prove the main result in this section.
\begin{theorem}\label{thm:wideproj}
  Let $\AA$ be a length abelian category. Then the following hold.
  \begin{enumerate}
    \item $\WWW(\WW) = \WW$ holds for a wide subcategory $\WW$ of $\AA$.
    \item $\WWW(\EE)$ is a wide subcategory of $\AA$ for a left Schur subcategory $\EE$ of $\AA$.
    \item The following diagram commutes, and all the horizontal maps are bijective.
    \[
    \begin{tikzcd}
      \wide\AA \rar["\simp"] \ar[dd, bend right=75, "1"'] \dar[hookrightarrow] & \sbrick\AA \lar["\Filt", shift left]\dar[hookrightarrow] \ar[dd, bend left=75, "1"]\\
      \lSchur\AA \rar["\simp", shift left] \dar["\WWW", twoheadrightarrow] & \mbrick\AA \dar["\mmax", twoheadrightarrow] \lar["\Filt", shift left]  \\
       \wide\AA \rar["\simp"] & \sbrick\AA \lar["\Filt", shift left]
    \end{tikzcd}
    \]
  \end{enumerate}
\end{theorem}
\begin{proof}
  (1)
  Clear from the definition of $\WWW(\WW)$ since $\WW$ is closed under cokernels.

  (2), (3)
  Let $\EE$ be a left Schur subcategory of $\AA$ and put $\MM:=\simp\EE$. By Theorem \ref{thm:main}, it clearly suffices to show that $\WWW(\EE) = \Filt(\mmax \MM)$, since $\mmax \MM$ is a semibrick by Proposition \ref{prop:maxsb}.
  By Lemma \ref{lem:we} (1), we have $\mmax\MM \subset \WWW(\EE)$. Since $\EE$ is extension-closed, it can be easily checked that Lemma \ref{lem:we} (2) implies $\Filt(\mmax\MM) \subset \WWW(\EE)$.

  Conversely, take any $M$ in $\WWW(\EE)$. We will prove $M \in \Filt(\mmax\MM)$ by the induction on the length $n$ of a $\MM$-filtration of $M$.

  If $n = 1$, we have $M \in \MM = \simp\EE$. Thus Lemma \ref{lem:we} (1) implies $M \in \mmax\MM$, hence in particular $M \in \Filt(\mmax\MM)$.
  Suppose $n>1$, then there is a short exact sequence in $\AA$
  \[
  \begin{tikzcd}
    0 \rar & L \rar & M \rar & N \rar & 0
  \end{tikzcd}
  \]
  such that $L$ and $N$ have $\MM$-filtrations of length smaller than $n$. In particular, we have $L,N \in \EE$. Then Lemma \ref{lem:we} (2) implies that $L,N \in \WWW(\EE)$ by $M \in \WWW(\EE)$. By the induction hypothesis, we have $L,N \in \Filt(\mmax\MM)$, which shows $M \in \Filt(\mmax\MM)$.
\end{proof}

\begin{remark}\label{rem:label}
  Let $\FF$ be a torsion-free class in $\AA$. Then $\mmax(\simp\FF) = \simp \WWW(\FF)$ holds by Theorem \ref{thm:wideproj}. In \cite[Proposition 6.5]{AP}, it is shown that this coincides with the set of \emph{brick labels of arrows} starting at $\FF$, which is introduced in \cite{DIRRT}. Therefore, in our context, considering the brick labels of $\FF$ is nothing but taking the maximal elements of the simple objects in $\FF$.
  In \cite{asai}, a bijection between functorially finite torsion-free classes and semibricks satisfying some condition was established, and this bijection is given by taking brick labels of arrows starting at $\FF$, thus coincides with taking $\mmax(\simp\FF)$.
\end{remark}

\section{Applications}\label{sec:5}
In this section, we give an application of the theory of monobricks to torsion-free classes and wide subcategories. We give new proofs of several results on these subcategories, such as Demonet-Iyama-Jasso's finiteness results \cite{DIJ} and Marks-\v{S}t\!'ov\'{i}\v{c}ek's bijection \cite{MS}, and make the relation between torsion-free classes and wide subcategories more transparent. Our result can be applied to any length abelian category, without using $\tau$-tilting theory.

\subsection{Maps between torsion-free classes and wide subcategories via monobricks}
In this section, we consider the restrictions of our maps $\WWW$ and $\FFF$ to $\WWW \colon \torf\AA \to \wide\AA$ and $\FFF \colon \wide\AA \to \torf\AA$.
By using monobricks, we can reprove and generalize a result by Marks-\v{S}t\!'ov\'{i}\v{c}ek only using an easy poset theoretical argument.
\begin{proposition}\label{prop:msid}
  Let $\AA$ be a length abelian category. Then the following hold.
  \begin{enumerate}
    \item Let $\MM$ be a monobrick in $\AA$ and $\NN$ a cofinal extension of $\MM$. Then $\mmax\MM = \mmax \NN$ holds. In particular, $\mmax \ov{\MM} = \mmax \MM$ holds.
    \item Let $\SS$ be a semibrick. Then we have $\mmax \ov{\SS} = \SS$. Thus the composition $\mmax \circ(\ov{-})\colon \sbrick\AA \to \sbrick\AA$ is the identity.
    \item \cite[Proposition 3.3]{MS} The composition $\WWW \circ \FFF \colon \wide\AA \to \wide\AA$ is the identity.
  \end{enumerate}
\end{proposition}
\begin{proof}
  (1)
  Let $M$ be a maximal element of $\MM$. If $M$ is not maximal in $\NN$, then there is some $N \in \NN$ with $M < N$. However, since $\MM$ is cofinal in $\NN$, there is some $M' \in \MM$ with $N \leq M'$, which implies $M < M'$. This is a contradiction, thus we have $\mmax \MM \subset \mmax \NN$.
  Conversely, let $N$ be a maximal element of $\NN$. Then since $\MM$ is cofinal in $\NN$, there is some $M\in\MM$ with $N \leq M$. Then the maximality implies $N = M \in \MM$, thus $N \in \mmax \MM$ holds.

  (2)
  Obvious from (1).

  (3)
  This follows from (2) and Theorems \ref{thm:torfproj} and \ref{thm:wideproj}.
\end{proof}
In general, the map $\FFF \colon \wide\AA \to \torf\AA$ is not a surjection, and its image is studied in \cite{AP}. We give a description of its image in terms of monobricks.
\begin{proposition}\label{prop:widetorf}
  Let $\FF$ be a torsion-free class in $\AA$. Then the following are equivalent.
  \begin{enumerate}
    \item $\FF = \FFF(\WWW(\FF))$ holds.
    \item There is a wide subcategory $\WW$ satisfying $\FF = \FFF(\WW)$.
    \item There is a semibrick $\SS$ satisfying $\simp\FF = \ov{\SS}$.
    \item $\simp\FF$ is a cofinal extension of some semibrick.
    \item $\mmax(\simp\FF)$ is cofinal in $\simp\FF$, that is, for every element $M$ in $\simp\FF$, there is an element $S \in \simp\FF$ such that $M \leq S$ holds and $S$ is maximal in $\simp\FF$.
  \end{enumerate}
\end{proposition}
\begin{proof}
  (1) $\Rightarrow$ (2):
  Obvious.

  (2) $\Rightarrow$ (3):
  Clear from Theorem \ref{thm:wideproj} and Theorem \ref{thm:torfproj}.

  (3) $\Rightarrow$ (4):
  This is clear since $\simp\FF = \ov{\SS}$ is a cofinal extension of $\SS$.

  (4) $\Rightarrow$ (5):
  Let $\SS$ be a semibrick such that $\simp\FF$ is a cofinal extension of $\SS$. Then we have $\SS = \mmax\SS = \mmax(\simp\FF)$ holds by Propositions \ref{prop:sbmax} and \ref{prop:msid}. Thus $\mmax(\simp\FF)$ is cofinal in $\simp\FF$.

  (5) $\Rightarrow$ (1):
  (5) implies that $\simp\FF$ is a cofinal extension of $\mmax(\simp\FF)$, and $\simp\FF$ is cofinally closed by Theorem \ref{thm:torfproj}. Thus we have $\ov{\mmax(\simp\FF)} = \simp\FF$ holds by Corollary \ref{cor:closure}.
  This is nothing but (1) under the bijections in Theorem \ref{thm:torfproj} and Theorem \ref{thm:wideproj}.
\end{proof}

\begin{example}
  Let us consider the 2-Kronecker case, see Example \ref{ex:kro} for details and undefined notation. Here, all cofinally closed monobricks $\MM$ except case (M2) satisfy that $\mmax\MM$ is cofinal in $\MM$. Thus all torsion-free classes except case (M2) (the torsion-free classes consisting of all preprojective modules) belong to the image of $\FFF \colon \wide\AA \to \torf\AA$.
\end{example}

As a corollary, we can quickly prove a bijection by Marks-\v{S}t\!'ov\'{i}\v{c}ek  (c.f. \cite[Corollary 3.11]{MS}).
\begin{corollary}\label{cor:brickfinbij}
  Let $\AA$ be a length abelian category. Then the maps $\FFF \colon  \wide \AA \to \torf\AA$ and $\WWW \colon \torf \AA \to \wide\AA$ induce a bijection between $\wide\AA$ and $\FFF(\wide\AA)$. If $\AA$ has only finitely many torsion-free classes, then $\FFF(\wide\AA) = \torf\AA$ holds, thus $\FFF$ and $\WWW$ are mutually inverse bijections between $\wide\AA$ and $\torf\AA$.
\end{corollary}
\begin{proof}
  Since the composition $\wide\AA \xrightarrow{\FFF} \torf\AA \xrightarrow{\WWW} \wide\AA$ is the identity by Proposition \ref{prop:msid}, it suffices to show the last assertion.
  Suppose that $\AA$ has finitely many torsion-free classes, and it suffices to prove that $\mmax \MM$ is cofinal in $\MM$ for every (cofinally closed) monobrick by Proposition \ref{prop:widetorf}.
  We will see in Theorem \ref{thm:bfinite} that there are only finitely many bricks in $\AA$ up to isomorphism. Therefore, every monobrick $\MM$ is a finite poset, thus clearly $\mmax\MM$ is cofinal in $\MM$.
\end{proof}

\subsection{Finiteness conditions}
In this subsection, we study several finiteness conditions on monobricks.
First we consider when $\torf\AA$ or $\wide\AA$ or $\mbrick\AA$ is finite.
We denote by $\brick\AA$ the set of isomorphism classes of bricks in $\AA$.
\begin{theorem}\label{thm:bfinite}
  Let $\AA$ be a length abelian category. Then the following are equivalent:
  \begin{enumerate}[font=\upshape]
    \item[(1)] $\brick\AA$ is finite, that is, there are only finitely many bricks in $\AA$ up to isomorphism.
    \item[(2)] $\mbrick\AA$ is finite.
    \item[(2)$^\prime$] $\lSchur\AA$ is finite.
    \item[(3)] $\ccmbrick\AA$ is finite.
    \item[(3)$^\prime$] $\torf\AA$ is finite.
    \item[(4)] $\sbrick\AA$ is finite.
    \item[(4)$^\prime$] $\wide\AA$ is finite.
    \item[(5)] There are only finitely many subcategories of $\AA$ which are closed under extensions, kernels and images.
  \end{enumerate}
\end{theorem}
\begin{proof}
  First, note that (i) and (i)$^\prime$ are equivalent for $i=2,3,4$ by Theorems \ref{thm:main} and \ref{thm:torfproj}.

  (1) $\Rightarrow$ (2):
  This is clear since $\mbrick\AA$ is a subset of $2^{\brick\AA}$, the power set of $\brick\AA$.

  (2) $\Rightarrow$ (3):
  This is clear by $\ccmbrick\AA \subset \mbrick\AA$.

  (3) $\Rightarrow$ (4):
  This is clear by the injection $(\ov{-})\colon \sbrick\AA \hookrightarrow \ccmbrick\AA$ shown in  Proposition \ref{prop:msid}.

  (4) $\Rightarrow$ (1):
  The map $\brick\AA \to \sbrick\AA$ defined by $S \mapsto \{S\}$ is clearly injective.

  (2)$^\prime$ $\Rightarrow$ (5):
  This is clear since every subcategory of $\AA$ closed under extensions, kernels and images is left Schur by Proposition \ref{prop:0kerlschur}.

  (5) $\Rightarrow$ (3)$^\prime$:
  This is clear since every torsion-free class in $\AA$ is closed under extensions, kernels and images.
\end{proof}

\begin{definition}
  We call a length abelian category $\AA$ \emph{brick-finite} if it satisfies the equivalent conditions of Theorem \ref{thm:bfinite}.
\end{definition}

\begin{remark}
  In the case $\AA = \mod \Lambda$ for a finite-dimensional algebra $\Lambda$, the equivalence of (1) and (3)$^\prime$ is a particular case of \cite[Theorems 3.8, 4.2]{DIJ}, and such an algebra is called \emph{$\tau$-tilting finite}. Actually it was shown in \cite{DIJ} that $\mod\Lambda$ is brick-finite if and only if there are only finitely many \emph{functorially finite} torsion-free classes, a little stronger result than ours.
\end{remark}

Next we consider when each monobrick consists of finitely many bricks.
We begin with the following general observation on posets. A subset $X$ of a poset $P$ is called a \emph{chain} if $X$ is totally ordered, and an \emph{antichain} if every two distinct elements in $X$ are incomparable. For an element $m$ of a poset $P$, we put $\downarrow m := \{ x \in P \, | \, x \leq m \}$.
\begin{lemma}\label{lem:finposet}
  Let $P$ be a poset such that every chain in $\downarrow m$ is finite for every $m$ in $P$. Then $P$ is finite if and only if it satisfies the following two conditions.
  \begin{enumerate}
    \item $\mmax P$ is cofinal in $P$, that is, every element is below some maximal element.
    \item Every antichain of $P$ is a finite set.
  \end{enumerate}
\end{lemma}
\begin{proof}
  If $P$ is finite, then it clearly satisfies (1) and (2).

  Conversely, suppose that $P$ is an infinite set.
  Since $\mmax P$ is an antichain of $P$, it is a finite set by (2). By (1), we have $P = \bigcup_{m \in \mmax P} (\downarrow m)$. Since $P$ is infinite and $\mmax P$ is finite, we may assume that $\downarrow m_1$ is an infinite set for $m_1 \in \mmax P$.

  Put $P_1:= (\downarrow m_1) \setminus \{ m_1 \} = \{ x \in P \, | \, x < m_1 \}$. Clearly $P_1$ also satisfies (2), since every antichain of $P_1$ is also an antichain of $P$.
  Suppose that there is an element $x \in P_1$ which is not below any maximal element in $P_1$. Since $x$ is not maximal in $P_1$, there is some $x < x_1$ with $x_1 \in P_1$, and $x_1$ is not below any maximal element in $P_1$. By iterating this, we obtain an infinite chain inside $P_1 \subset (\downarrow m_1)$, which is a contradiction. Thus $P_1$ satisfies (1).

  Now we can find an element $m_2 \in P_1$ such that $\downarrow m_2$ is infinite. Then we apply the same process to $P_2 := \{ x \in P \, | \, x < m_2 \}$. We can iterate this procedure, and obtain an infinite chain $m_1 > m_2 > m_3 > \cdots$ in $\downarrow m_1$. This is a contradiction.
\end{proof}

By using this, we can prove the following criterion on finiteness of a monobrick.
\begin{proposition}\label{prop:finmbrick}
  Let $\MM$ be a monobrick in $\AA$. Then $\MM$ is a finite set if and only if it satisfies the following two conditions:
  \begin{enumerate}
    \item Every element in $\MM$ is below some maximal element in $\MM$.
    \item Every semibrick $\SS$ with $\SS \subset \MM$ is a finite set.
  \end{enumerate}
\end{proposition}
\begin{proof}
  First we will check that the poset $\MM$ satisfies the assumption in Lemma \ref{lem:finposet}. Take any $M \in \MM$ and consider $\downarrow M$. Clearly $0 \leq l(X) < l(M)$ holds for every element $X \neq M$ in $\downarrow M$, where $l(-)$ denotes the lengths of objects in $\AA$.
  If $X < X'$ in $\MM$, then $l(X) < l(X')$ holds. Thus clearly $\downarrow M$ cannot contain any infinite chains.

  The conditions (1) and (2) in Lemma \ref{lem:finposet} are nothing but (1) and (2) in this proposition respectively. In particular, we can check that a subset $\SS$ of $\MM$ is an antichain if and only if $\SS$ is a semibrick. Thus the assertion holds.
\end{proof}

As an application, we have the following criterion on finiteness of the number of simple objects in a given torsion-free class.
\begin{corollary}\label{cor:torfsimpfin}
  Let $\FF$ be a torsion-free class in $\AA$. Then $\simp\FF$ is a finite set if and only if $\FF$ satisfies the following conditions.
  \begin{enumerate}
    \item $\FF = \FFF(\WW)$ holds for some wide subcategory $\WW$ of $\AA$ (or equivalently, the equivalent conditions in Proposition \ref{prop:widetorf} are satisfied).
    \item Every semibrick $\SS$ satisfying $\SS \subset \simp\FF$ is finite.
  \end{enumerate}
\end{corollary}

The following is the fundamental relation between brick-finiteness and the finiteness of each monobrick.
\begin{theorem}\label{thm:2ndfin}
  Let $\AA$ be a length abelian category. Then the following are equivalent.
  \begin{enumerate}[font = \upshape]
    \item[(1)] Every monobrick in $\AA$ is a finite set.
    \item[(2)] Every cofinally closed monobrick in $\AA$ is a finite set.
    \item[(2)$^\prime$] $\#\simp\FF$ is finite for every torsion-free class $\FF$ in $\AA$.
    \item[(3)] Every semibrick in $\AA$ is a finite set, and the map $\FFF\colon \wide\AA \hookrightarrow \torf\AA$ is surjective.
  \end{enumerate}
  Moreover, if $\AA = \mod\Lambda$ for a finite-dimensional algebra $\Lambda$, then the following is also equivalent.
  \begin{enumerate}[font = \upshape]
    \item[(4)] $\mod\Lambda$ is brick-finite, that is, there are only finitely many bricks in $\mod\Lambda$ up to isomorphism.
  \end{enumerate}
\end{theorem}
\begin{proof}
  (1) $\Rightarrow$ (2): Obvious.

  (2) $\Rightarrow$ (1): Let $\MM$ be a monobrick in $\AA$. Then $\MM \subset \ov{\MM}$ holds for the cofinal closure of $\MM$. Since $\ov{\MM}$ is cofinally closed, it is finite by (2), thus so is $\MM$.

  (2) $\Leftrightarrow$ (2)$^\prime$: Clear from Theorem \ref{thm:torfproj}.

  (1) + (2)$^\prime$ $\Rightarrow$ (3):
  The surjectivity of the map $\FFF\colon\wide\AA \hookrightarrow \torf\AA$ follows from Corollary \ref{cor:torfsimpfin}. Since every semibrick is a monobrick, it is finite by (1).

  (3) $\Rightarrow$ (2)$^\prime$:
  Clear from Corollary \ref{cor:torfsimpfin}.

  Now we have shown the equivalence of (1), (2), (2)$^\prime$ and (3).
  From now on, suppose that $\Lambda$ is a finite-dimensional algebra and $\AA = \mod\Lambda$.

  (2)$^\prime$ $\Rightarrow$ (4):
  Suppose that $\mod\Lambda$ is not brick-finite. Then by \cite[Theorem 3.8]{DIJ}, there is a torsion-free class $\FF$ in $\mod\Lambda$ which is not functorially finite. Put $\FF_0:= 0 \in \torf(\mod\Lambda)$.
  Then \cite[Theorem 3.1]{DIJ} implies that there is a functorially finite torsion-free class $\FF_1$ satisfying $\FF_0 \subsetneq \FF_1 \subset \FF$. Since $\FF$ is not functorially finite, we have $\FF_1 \subsetneq \FF$.
  By repeating this process, we obtain a strictly ascending chain $0 = \FF_0 \subsetneq \FF_1 \subsetneq \FF_2 \subsetneq \cdots$ of torsion-free classes.
  Put $\GG := \bigcup_{i\geq 0} \FF_i$. Then it is clearly a torsion-free class, and $\simp\GG$ is finite by (2)$^\prime$.
  Therefore, there is some $i$ such that $\simp\GG \subset \FF_i$ holds. Since $\FF_i$ is extension-closed, this would imply $\GG = \Filt(\simp\GG) \subset \FF_i \subset \GG$, thus $\FF_i = \FF_{i+1} = \cdots =  \GG$, which is a contradiction.

  (4) $\Rightarrow$ (1):
  Clear.
\end{proof}

We propose the following conjecture related to this, which is of interest in its own.
\begin{conjecture}
  Let $\Lambda$ be a finite-dimensional algebra. If every semibrick in $\mod\Lambda$ is a finite set, then $\mod\Lambda$ is brick-finite, that is, $\Lambda$ is $\tau$-tilting finite.
\end{conjecture}
Roughly speaking, Proposition \ref{prop:finmbrick} and Theorem \ref{thm:2ndfin} say that in order to show brick-finiteness, we have to show the finiteness of antichains (semibricks) and chains of bricks. Thus this conjecture is roughly equivalent to the following question: if every monobrick has \emph{finite width} (finite antichains), then does every monobrick have a \emph{finite height}?

Regarding this, it was recently shown in \cite[Theorem 1.1]{ST} that the finiteness of \emph{height} implies brick-finiteness. More precisely, it was shown that if there is an upper bound on the lengths of bricks, then $\mod\Lambda$ is brick-finite.

\section{Monobricks over Nakayama algebras}\label{sec:6}
In this section, we fix an algebraically closed base field $k$.
For a finite-dimensional algebra $\Lambda$, we put $\mbrick\Lambda:=\mbrick(\mod\Lambda)$ and so on.
In this section, we investigate monobricks and left Schur subcategories of $\mod\Lambda$ for a Nakayama algebra $\Lambda$. For details on Nakayama algebras, we refer the reader to standard texts such as \cite[V.3]{ASS}.

First of all, we show that left Schur subcategories are precisely subcategories closed under extensions, kernels and images.
\begin{theorem}\label{thm:nakeic}
  Let $\Lambda$ be a Nakayama algebra and $\EE$ a subcategory of $\mod\Lambda$. Then $\EE$ is left Schur if and only if $\EE$ is closed under extensions, kernels and images. In particular, we have a bijection between the following two sets:
  \begin{enumerate}
    \item $\mbrick\Lambda$, the set of monobricks in $\mod\Lambda$.
    \item The set of subcategories of $\mod\Lambda$ closed under extensions, kernels and images.
  \end{enumerate}
  The maps are given by $\Filt$ and $\simp$.
\end{theorem}
\begin{proof}
  We use the result in \cite[Corollary 5.19]{eno:jhp}: every torsion-free class in $\mod\Lambda$ satisfies the Jordan-H\"older property. We refer the reader to \cite{eno:jhp} for details on this property.

  By Proposition \ref{prop:0kerlschur}, we only have to show that every left Schur subcategory $\EE$ of $\mod\Lambda$ is closed under kernels and images. By Theorem \ref{thm:main}, there is a monobrick $\MM$ satisfying $\EE = \Filt\MM$.
  Consider the cofinal closure $\ov{\MM}$ of $\MM$ and put $\FF:= \FFF(\EE)$. Then we have $\ov{\MM} = \simp \FF \supset \MM$ by Proposition \ref{prop:ccmap}.
  Let $f \colon X \to Y$ be a map in $\EE$. Since $\FF$ is closed under kernels and images in $\mod\Lambda$, we obtain the following short exact sequence in $\FF$:
  \[
  \begin{tikzcd}
    0 \rar & \ker f \rar & X \rar & \im f \rar & 0.
  \end{tikzcd}
  \]
  Since $\FF$ satisfies the Jordan-H\"older property, we can speak of \emph{composition factors inside $\FF$}. Since $X$ is in $\EE = \Filt\MM$, all the composition factors of $X$ inside $\FF$ belong to $\MM$ by $\MM \subset \simp\FF$.
  Therefore, all the composition factors of $\im f$ and $\ker f$ must be in $\MM$, since the above short exact sequence is a conflation in $\FF$. This implies that $\im f$ and $\ker f$ belong to $\Filt\MM = \EE$.
\end{proof}

Our next aim is to give a combinatorial classification of monobricks for Nakayama algebras.
The following basic observation on quotient algebras and monobricks is useful. Recall that for a two-sided ideal $I$ of a finite-dimensional algebra $\Lambda$, we have the natural fully faithful functor $\mod(\Lambda/I) \hookrightarrow \mod\Lambda$, and its essential image consists of $\Lambda$-modules $M$ satisfying $MI = 0$. Using this, we may identify $\mod (\Lambda/I)$ with the subcategory of $\mod\Lambda$ consisting of such modules.
\begin{proposition}
  Let $\Lambda$ be a finite-dimensional algebra and $I$ a two-sided ideal of $\Lambda$. Then by identifying $\mod(\Lambda/I)$ with a subcategory of $\mod\Lambda$, we have
  \[
  \mbrick (\Lambda/I) = \mod(\Lambda/I) \cap \mbrick \Lambda
  \]
\end{proposition}
\begin{proof}
  This follows from the fact that the natural functor $\mod(\Lambda/I) \hookrightarrow \mod\Lambda$ is fully faithful and that a morphism in $\mod(\Lambda/I)$ is an injection in $\mod(\Lambda/I)$ if and only if so it is in $\mod\Lambda$.
\end{proof}
As a consequence, the classification of monobricks over $\Lambda/I$ can be obtianed over $\Lambda$.
Keeping this in mind, it suffices to consider the following two classes of Nakayama algebras.
\begin{definition}
  Let $n$ be a positive integer. Then we define two algebras $A_n$ and $B_n$ as follows:
  \begin{enumerate}
    \item $A_n$ is the path algebra of the following quiver.
    \[
    \begin{tikzcd}
      1 & 2 \lar & \cdots \lar & n \lar
    \end{tikzcd}
    \]
    \item $B_n$ is the quotient of the path algebra of the following quiver by the ideal generated by all the paths of length $n$.
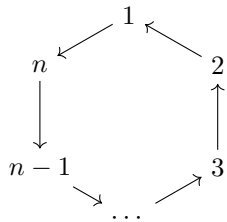
    \[
    \begin{tikzpicture}[scale=0.5]
       \newdimen\Ra
       \Ra=2.7cm
       \node (1) at (90:\Ra)  {$1$};
       \node (2) at (30:\Ra)  {$2$};
       \node (3) at (-30:\Ra)  {$3$};
       \node (4) at (-90:\Ra)  {$\cdots$};
       \node (5) at (210:\Ra)  {$n-1$};
       \node (n-1) at (150:\Ra)  {$n$};

      \draw[->] (2) to (1);
      \draw[->] (4) to (3);
      \draw[->] (3) to (2);
      \draw[->] (5) to (4);
      \draw[->] (n-1) to (5);
      \draw[->] (1) to (n-1);
    \end{tikzpicture}
    \]
  \end{enumerate}
\end{definition}
Note that we have the natural identification $B_n/\la e_n \ra \iso A_{n-1}$, where $e_n$ is the primitive idempotent of $B_n$ corresponding to $n$.

The following shows that to classify monobricks over Nakayama algebras, it suffices to consider $A_n$ and $B_n$.
\begin{proposition}
  Let $\Lambda$ be a basic connected Nakayama algebra with $\#\simp(\mod\Lambda) = n$.
  \begin{enumerate}
    \item If the quiver of $\Lambda$ is acyclic, then $\Lambda \iso A_n/I$ for some $I$, thus $\mbrick\Lambda \subset \mbrick A_n$ holds.
    \item If the quiver of $\Lambda$ is cyclic, then there exist a Nakayama algebra $B'$ and two algebra surjections $B' \defl \Lambda$ and $B'\defl B_n$ such that $\mbrick \Lambda \subset \mbrick B_n$ holds inside $\mbrick B'$.
  \end{enumerate}
\end{proposition}
\begin{proof}
  (1)
  This is well-known, e.g. \cite[Theorem V.3.2]{ASS}.

  (2)
  The existence of a Nakayama algebra $B'$ such that $\Lambda$ and $B_n$ are quotients of $B'$ is obvious (consider the path algebra of the cyclic quiver and annihilate sufficiently large paths), thus it suffices to see that every brick $M$ in $\mod B'$ is contained in $\mod B_n$. This is clear since if an indecomposable module $M$ does not belong to $B_n$, then it is easily checked that $M$ has a non-zero endomorphism which is not an isomorphism.
\end{proof}

To deal with modules over $A_n$ and $B_n$, we will use the following combinatorial description.
\begin{definition}
  Let $n$ be a positive integer.
  \begin{itemize}
    \item We put $[n]:= \{ 1,2,\dots,n\}$.
    \item For two elements $i,j$ in $[n]$, we calculate $i+j \in [n]$ and $i-j \in [n]$ modulo $n$, for example, $n+1 = 1$ and $1 - 1 = n$.
    \item An \emph{arc on $[n]$} is an element of $[n] \times [n]$.
    \item An \emph{admissible arc on $[n]$} is an arc $(i,j)$ satisfying $i < j$.
    \item For an arc $\al = (i,j)$ on $[n]$, we call $i$ its \emph{starting point} and $j$ its \emph{ending point}.
    \item The \emph{socle series} of an arc $\al = (i,j)$ on $[n]$ is a sequence of elements in $[n]$ defined by $(i,i+1,\dots,j-1)$.
    \item An \emph{arc diagram $\DD$ on $[n]$} is a set of arcs, that is, a subset of $[n]\times [n]$.
    \item An arc diagram $\DD$ is \emph{admissible} if every arc in $\DD$ is admissible.
  \end{itemize}
\end{definition}
We represent arcs and arc diagrams on $[n]$ as follows:
Consider the Euclidean plane $\R^2$ and put $i$ on $(i,0) + \Z (n,0)$ for each $i \in [n]$.
Then for an arc $\al = (i,j)$, we draw ``arcs" in the upper half-plane which connect each $i$ with the first $j$ which appears right to this $i$.

For example, the following is the arc diagram $\DD = \{ (1,1), (2,3), (3,2) \}$ on $[3]$. These three arcs have socle series $(1,2,3), (2), (3,1)$ respectively.
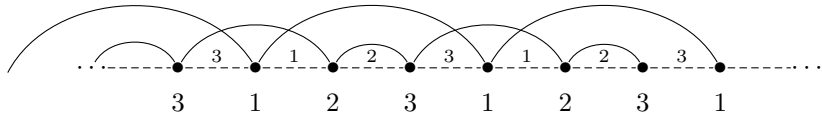
\[
\begin{tikzcd}[row sep=0mm, nodes={inner sep=0pt}]
  \ar[-,rrr, bend left = 60]
  &\cdots \ar[-,r, bend left = 60] \rar[-,dashed]
  & \bullet \ar[loop below,phantom, "3"] \ar[-,rr, bend left = 60] \rar[-,dashed, "3"]
  & \bullet \ar[loop below,phantom, "1"] \ar[-,rrr, bend left = 60] \rar[-,dashed, "1"]
  & \bullet \ar[loop below,phantom, "2"]\ar[-,r, bend left = 60] \rar[-,dashed, "2"]
  & \bullet\ar[loop below,phantom, "3"] \ar[-,rr, bend left = 60] \rar[-,dashed, "3"]
  & \bullet \ar[loop below,phantom, "1"]\ar[-,rrr, bend left = 60] \rar[-,dashed, "1"]
  & \bullet \ar[loop below,phantom, "2"]\ar[-,r, bend left = 60] \rar[-,dashed, "2"]
  & \bullet\ar[loop below,phantom, "3"] \rar[-,dashed, "3"]
  & \bullet\ar[loop below,phantom, "1"] \rar[-,dashed]
  & \cdots
\end{tikzcd}
\]
As in this figure, it is convenient to draw a dashed line on the $x$-axis and label each line segment as above, so that the socle series of $\al$ is the sequence of labels surrounded by $\al$.
Also, we often draw an admissible arc diagram by omitting the repeated part, for example, the following is a picture of the admissible arc diagram $\{(1,2),(1,4),(3,4)\}$ on $[4]$.
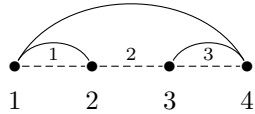
\[
\begin{tikzcd}[row sep=0mm, nodes={inner sep=0pt}]
  & \bullet \ar[loop below,phantom, "1"] \ar[-,r, bend left = 60]\ar[-,rrr, bend left = 60] \rar[-,dashed, "1"]
  & \bullet \ar[loop below,phantom, "2"] \rar[-,dashed, "2"]
  & \bullet\ar[loop below,phantom, "3"] \ar[-,r, bend left = 60] \rar[-,dashed, "3"]
  & \bullet \ar[loop below,phantom, "4"]
\end{tikzcd}
\]

We say that a sequence $(n_1,\dots,n_a)$ is a \emph{partial sequence} of a sequence $(m_1,\dots,m_b)$ if there is some integer $i$ with $1 \leq i \leq b-a +1$ satisfying $n_1 = m_i, n_2 = m_{i+1}, \dots, n_a = m_{i+a-1}$.
For example, $(3,1)$, $(2)$ and $(2,3,1)$ are partial sequences of $(2,3,1)$, but $(1,2)$, $(2,1)$ and $(3,1,2)$ are not.
\begin{definition}
  We say that a pair $\{\al,\be\}$ of two different arcs $\al$ and $\be$ on $[n]$ is a \emph{weakly non-crossing pair} if \emph{either} of the following conditions is satisfied:
  \begin{itemize}
    \item The socle series of $\al$ is a partial sequence of that of $\be$.
    \item The socle series of $\be$ is a partial sequence of that of $\al$.
    \item The socle series of $\al$ and $\be$ are disjoint, that is, there exists no element in $[n]$ which appears in both socle series.
  \end{itemize}
  Moreover, for a weakly non-crossing pair $\{\al,\be\}$, we define the following.
  \begin{enumerate}
    \item It is a \emph{mono-crossing pair} if $\al$ and $\be$ have the same starting point.
    \item It is an \emph{epi-crossing pair} if $\al$ and $\be$ have the same ending point.
    \item It is a \emph{non-crossing pair} if it is neither mono-crossing nor epi-crossing.
  \end{enumerate}
  We say that $\{\al,\be\}$ is a \emph{strictly crossing pair} if it is not weakly non-crossing.
\end{definition}
Intuitively, a pair of two arcs is weakly non-crossing if the arcs do not \emph{cross} in the half-plane model except at their starting points or ending points, and it is non-crossing if in addition they do not have the same starting points or ending points.
\begin{example}
  Consider the arc diagram $\{(1,1),(2,3),(3,1),(3,2) \}$ on $[3]$:
  
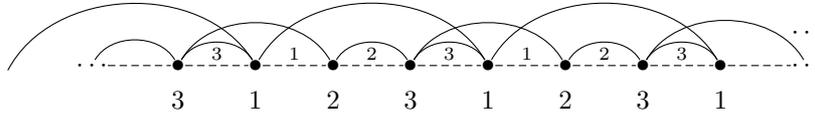
\[
  \begin{tikzcd}[row sep=0mm, nodes={inner sep=0pt}]
    \ar[-,rrr, bend left = 60]
    &\cdots \ar[-,r, bend left = 60] \rar[-,dashed]
    & \bullet \ar[loop below,phantom, "3"] \ar[-,rr, bend left = 60] \ar[-,r, bend left = 60] \rar[-,dashed, "3"]
    & \bullet \ar[loop below,phantom, "1"] \ar[-,rrr, bend left = 60] \rar[-,dashed, "1"]
    & \bullet \ar[loop below,phantom, "2"]\ar[-,r, bend left = 60] \rar[-,dashed, "2"]
    & \bullet\ar[loop below,phantom, "3"] \ar[-,rr, bend left = 60] \ar[-,r, bend left = 60] \rar[-,dashed, "3"]
    & \bullet \ar[loop below,phantom, "1"]\ar[-,rrr, bend left = 60] \rar[-,dashed, "1"]
    & \bullet \ar[loop below,phantom, "2"]\ar[-,r, bend left = 60] \rar[-,dashed, "2"]
    & \bullet\ar[loop below,phantom, "3"] \ar[-,rr, bend left = 60]\ar[-,r, bend left = 60] \rar[-,dashed, "3"]
    & \bullet\ar[loop below,phantom, "1"] \rar[-,dashed]
    & \cdots \ar[loop above,phantom, "\cdots"]
  \end{tikzcd}
  \]
  Then the crossing relations between the four arcs are as follows:
  
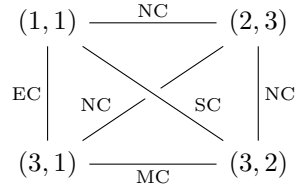
\[
  \begin{tikzcd}[sep=huge]
    (1,1) \rar[-,"\text{NC}"]\dar[-,"\text{EC}"'] & (2,3) \dar[-,"\text{NC}"]\\
    (3,1) \rar[-,"\text{MC}"'] \ar[ru,-,"\text{NC}", near start] & (3,2) \ar[lu,-,"\text{SC}"', crossing over, near start]
  \end{tikzcd}
  \]
  Here NC, EC, MC and SC mean non-crossing, epi-crossing, mono-crossing and strictly crossing respectively.
\end{example}

\begin{remark}\label{rm:crossing}
  Suppose that $\al = (a,b)$ and $\be = (c,d)$ are distinct \emph{admissible} arcs. Then it is straightforward to see that $\{\al,\be\}$ is strictly crossing if and only if $a<c<b<d$ or $c<a<d<b$.
\end{remark}

\begin{definition}
  Let $n$ be a positive integer and $\DD$ an arc diagram on $[n]$.
  \begin{enumerate}
    \item $\DD$ is \emph{non-crossing} if every distinct pair of arcs in $\DD$ is a non-crossing pair.
    \item $\DD$ is \emph{mono-crossing} if every distinct pair of arcs in $\DD$ is either a mono-crossing or a non-crossing pair.
  \end{enumerate}
\end{definition}
Now let us return to the algebraic side.

\begin{definition}
  Let $n$ be a positive integer and $\al = (i,j)$ be an arc on $[n]$. Then we denote by $M_\al$ the unique indecomposable $B_n$-module satisfying $\soc M_\al = S_i$ and $\top M_\al = S_{j-1}$, where $S_i$ for $i \in [n]$ is the simple $B_n$-module corresponding to the vertex $i$. If $\al$ is admissible, that is, $i<j$, then we may regard $M_\al$ as an $A_{n-1}$-module by the surjection $B_n \defl B_n/\la e_n\ra \iso A_{n-1}$.
\end{definition}
Now by the standard description of indecomposable modules and morphisms between them over Nakayama algebras (e.g. \cite[Theorem V.3.5]{ASS}), it is easy to show the following.
\begin{proposition}\label{prop:nakayamacorresp}
  Let $n$ be a positive integer. Then the assignment $\al \mapsto M_\al$ induces a bijection between the set of arcs on $[n]$ and $\brick B_n$, and a bijection between the set of admissible arcs on $[n]$ and $\brick A_{n-1}$. Moreover, the following hold for two arcs $\al$ and $\be$ on $[n]$.
  \begin{enumerate}
    \item $\{\al,\be\}$ is a non-crossing pair if and only if $\{M_\al,M_\be\}$ is a semibrick.
    \item $\{\al,\be\}$ is a mono-crossing pair if and only if $\{M_\al,M_\be\}$ is a monobrick and not a semibrick.
  \end{enumerate}
  Therefore, we have the following bijections, where $\MM_\DD:= \{M_\al \, | \, \al \in \DD\}$ for an arc diagram $\DD$:
  \[
  \begin{tikzcd}
    \{ \text{mono-crossing arc diagrams on $[n]$} \} \rar["\MM_{(-)}", "\sim"'] & \mbrick B_n \\
    \{ \text{non-crossing arc diagrams on $[n]$} \} \rar["\sim"]\uar[hookrightarrow] & \sbrick B_n \uar[hookrightarrow]
  \end{tikzcd}
  \]
  and
  \[
  \begin{tikzcd}
    \{ \text{mono-crossing admissible arc diagrams on $[n]$} \} \rar["\MM_{(-)}", "\sim"'] & \mbrick A_{n-1} \\
    \{ \text{non-crossing admissible arc diagrams on $[n]$} \} \rar["\sim"]\uar[hookrightarrow] & \sbrick A_{n-1}. \uar[hookrightarrow]
  \end{tikzcd}
  \]
\end{proposition}
By combining this with Theorem \ref{thm:nakeic}, the problem of classifying all the subcategories closed under extensions, kernels and images is reduced to a purely combinatorial problem, namely, classifying all the mono-crossing (admissible) arc diagrams on $[n]$.

In the rest of this section, we give an explicit formula for $\#\mbrick A_n$ and $\#\mbrick B_n$. Note that a formula for $\#\sbrick A_n$ and $\#\sbrick B_n$ is given by Asai \cite[Lemmas 3.4, 3.7]{asai}:
\begin{align*}
  \#\sbrick A_{n-1} &= \frac{1}{n+1}\binom{2n}{n}
  & (\text{the $n$-th Catalan number, \cite[A000108]{OEIS}}) \\
  \#\sbrick B_n &= \binom{2n}{n} & (\text{\cite[A000984]{OEIS}})
\end{align*}
Here $\binom{n}{i}$ denotes the binomial coefficient.
We can also compute $\#\sbrick A_{n-1}$ using Proposition \ref{prop:nakayamacorresp}. Non-crossing admissible arc diagrams on $[n]$ clearly correspond to the classical \emph{non-crossing partitions on $[n]$} (see \cite[N.4.1]{rincon} for details). Therefore, its number is equal to the number of non-crossing partitions, which is well-known to be equal to the Catalan number.

Our enumeration of monobricks is based on the same idea: find a bijection between the set of mono-crossing arc diagrams and some combinatorial sets, whose cardinality has already been computed by combinatorialists.

The following is our enumerative result.
\begin{theorem}\label{thm:nakcount}
  Let $n$ be a positive integer. Then the following equalities hold.
  \begin{align}
    \#\mbrick A_n &= \text{\cite[A006318]{OEIS}}(n) = \sum_{i=0}^n \frac{1}{i+1}\binom{n}{i}\binom{n+i}{i}, \text{ the $n$-th large Schr\"oder number}. \label{eq1}\\
    \#\mbrick B_n &= \text{\cite[A002003]{OEIS}}(n) = 2 \sum_{i=0}^{n-1} \binom{n-1}{i} \binom{n+i}{i}. \label{eq2}
  \end{align}
\end{theorem}
\begin{proof}
  By Proposition \ref{prop:nakayamacorresp}, it suffices to count the numbers of
  \begin{enumerate}
    \item mono-crossing admissible arc diagrams on $[n+1]$, and
    \item mono-crossing arc diagrams on $[n]$.
  \end{enumerate}

  (1)
  We will show that the number of mono-crossing admissible arc diagrams on $[n]$ is equal to the $(n-1)$-th large Schr\"oder number.
  The outline of the enumeration is as follows: we will show that mono-crossing admissible arc diagrams are in bijection with \emph{non-crossing linked partitions} introduced in \cite{dyk}, whose number is known to be the large Schr\"oder number.

  A \emph{non-crossing linked partition of $[n]$} is a set $\pi$ of non-empty subsets of $[n]$ satisfying the following conditions.
  \begin{itemize}
    \item[(NCL1)] $[n] = \bigcup_{E \in \pi} E$ holds.
    \item[(NCL2)] For every $E,F \in \pi$ with $E \neq F$, there exists no $a < b < c < d$ satisfying $a,c \in E$ and $b,d \in F$.
    \item[(NCL3)] We have $\# (E \cap F) \leq 1$ for every distinct $E,F \in \pi$, and if $j \in E \cap F$, then either $j = \min E$, $\# E >1$ and $j \neq \min F$ hold, or the converse $j = \min F$, $\#F > 1$ and $j \neq \min E$ hold.
    In particular, if $E \cap F \neq \varnothing$, then $\# E, \#F >1$ holds.
  \end{itemize}
  Then the number of non-crossing linked partitions of $[n]$ is equal to the $(n-1)$-th large Schr\"oder number by \cite{dyk}. We will prove the equality (\ref{eq1}) by constructing a bijection from the set of non-crossing linked partitions of $[n]$ to the set of mono-crossing admissible arc diagrams.
  Our construction is essentially the same as the graphical presentation given in \cite{CWY}.

  Let $\pi$ be a non-crossing linked partition of $[n]$.
  For each $E\in\pi$ and $j \in E$ with $j \neq \min E$, we draw an arc $(\min E,j)$. By this, we obtain an admissible arc diagram $\DD_\pi$.

  We claim that $\DD_\pi$ is actually a mono-crossing arc diagram.
  Let $\{\al,\be\}$ be a pair of arcs in $\DD_\pi$ with $\al \neq \be$. By (NCL2) and Remark \ref{rm:crossing}, this pair is weakly non-crossing. Thus it suffices to show that $\{\al,\be\}$ is not epi-crossing.
  Assume that $\{\al,\be\}$ is epi-crossing, then $\al$ and $\be$ have the same ending point, that is, we can write $\al = (i_\al,j)$ and $\be = (i_\be,j)$ with $i_\al \neq i_\be$.
  By the construction of $\DD_\pi$, there is $E$ and $F$ in $\pi$ satisfying $\{i_\al,j\} \subset E$, $\{i_\be,j\} \subset F$, $i_\al = \min E$ and $i_\be = \min F$.
  Then $E \neq F$ holds by $i_\al \neq i_\be$. Now we have $j \in E \cap F$ but $j \neq \min E,\min F$, which contradicts (NCL3). Therefore, $\{\al,\be\}$ is not epi-crossing, thus $\DD_\pi$ is a mono-crossing arc diagram.

  Conversely, let $\DD$ be a mono-crossing admissible arc diagram on $[n]$. For each $i$ in $[n]$, define $E_i \subset [n]$ as follows:
  \[
  E_i :=
  \begin{cases}
    \{i \} \cup \{ j \, | \, (i,j) \in \DD \} & \text{if there is some arc starting at $i$,} \\
    \{ i \} & \text{if there is no arc either starting or ending at $i$,} \\
    \varnothing & \text{otherwise}
  \end{cases}
  \]
  Note that $i = \min E_i$ holds if $E_i \neq \varnothing$, thus the non-empty $E_i$'s are pairwise distinct.

  Put $\pi_\DD := \{E_i \, | \, 1 \leq i \leq n, E_i \neq \varnothing \}$. We claim that $\pi_\DD$ is a non-crossing linked partition of $[n]$.
  Clearly $\pi_\DD$ satisfies (NCL1). Assume that $\pi_\DD$ does not satisfy (NCL2), that is, there is some $E_{i_1},E_{i_2} \in \pi_\DD$ with $i_1 \neq i_2$ and $a,c \in E_{i_1}$, $b,d \in E_{i_2}$ satisfying $a<b<c<d$. Then we have $i_1 < b < c < d$ by $i_1 = \min E_{i_1}$.
  We consider the two cases $i_1 < i_2$ and $i_2 < i_1$.

  If $i_1 < i_2$, then $i_1 < i_2 \leq b < c < d$ holds by $i_2 = \min E_{i_2}$. Now $c \in E_{i_1}$ and $d \in E_{i_2}$ imply $(i_1,c) \in \DD$ and $(i_2,d) \in \DD$. From this, $(i_1,c)$ and $(i_2,d)$ are strictly crossing by $i_1 < i_2 < c< d$, which is a contradiction.
  If $i_2 < i_1$, then $i_2 < i_1 < b < c$ holds. Now $c \in E_{i_1}$ and $b \in E_{i_2}$ imply $(i_1,c) \in \DD$ and $(i_2,b) \in \DD$. Since these two arcs are strictly crossing, this is a contradiction.
  Thus (NCL2) holds.

  Next we will show that $\pi_\DD$ satisfies (NCL3). Suppose that $E_{i_1} \cap E_{i_2} \neq \varnothing$ for $i_1 \neq i_2$ and take $j \in E_{i_1} \cap E_{i_2}$. If $\# E_{i_1} = 1$, then $j = i_1$ and there is no arc either starting or ending at $i_1$.
  However, $i_1 \in E_{i_2}$ and $i_1 \neq i_2$ implies that $(i_2,i_1) \in \DD$, which is a contradiction. Thus $\# E_{i_1} > 1$ and $\# E_{i_2} > 1$ hold.
  Now if $j \neq i_1$ and $j \neq i_2$, then $(i_1,j), (i_2,j) \in \DD$ holds. This is a contradiction since these two arcs are epi-crossing. Thus either $j = i_1$ or $j = i_2$ holds, which implies $E_{i_1} \cap E_{i_2} \subset \{i_1, i_2\}$.
  If we have $E_{i_1} \cap E_{i_2} = \{i_1,i_2\}$, then $i_1 \in E_{i_2}$ and $i_2 \in E_{i_1}$ imply $(i_2,i_1) \in \DD$ and $(i_1,i_2) \in \DD$ respectively. Since $\DD$ is admissible, it follows that $i_2 < i_1 < i_2$, which is a contradiction.
  Thus exactly one of the cases $\{i_1\} = E_{i_1} \cap E_{i_2}$ and $\{i_2\} = E_{i_1} \cap E_{i_2}$ holds, hence $\# (E_{i_1} \cap E_{i_2}) \leq 1$.
  In the former case, we have $j = i_1 =\min E_{i_1}$ and $j = i_1 \neq i_2 = \min E_{i_2}$, and in the latter we have $j = \min E_{i_2}$ and $j \neq \min E_{i_1}$. Therefore (NCL3) is satisfied.

  Now we have shown that $\pi_\DD$ is a non-crossing linked partition of $[n]$ if $\DD$ is a mono-crossing admissible arc diagram. It is quite straightforward to see that $\DD = \DD_{\pi_\DD}$ holds for a mono-crossing admissible arc diagram, so we omit the proof.

  Finally we show that $\pi = \pi_{\DD_\pi}$ holds for a non-crossing linked partition $\pi$ of $[n]$. Let $E\in \pi$, and we will show $E = E_i$ for $i:= \min E$. We consider two cases:

  (Case 1): $E = \{ i \}$. Suppose that there is some arc $(i,j)$ in $\DD_\pi$. By the construction of $\DD_\pi$, there is some $F \in \pi$ with $i = \min F$ and $j \in F$. This contradicts (NCL3) since $i \in E \cap F$ and $\# E = 1$. It follows that there is no arc starting at $i$.
  Similarly, suppose that there is some arc $(j,i)$ in $\DD_\pi$. Then there is some $F \in \pi$ with $j = \min F$ and $i \in F$. This contradicts (NCL3) by $i \in E \cap F$ and $\#E = 1$.
  Therefore, there is no arc either starting or ending at $i$, hence $E_i = \{i\} = E$ holds.

  (Case 2): $\# E > 1$. In this case, there is some arc starting at $i$ in $\DD_\pi$. By construction, $E \subset E_i$ holds.
  Conversely, take $j \in E_i$ with $j \neq i$. Then $(i,j) \in \DD_\pi$ holds, thus there is some $F \in \pi$ with $i = \min F$ and $j \in F$. Since $i \in E \cap F$ and $i = \min E = \min F$, we must have $E = F$ by (NCL3).
  Thus $j \in F = E$ holds, hence $E = E_i$.

  We have shown $\pi \subset \pi_{\DD_\pi}$. Conversely, take $E_i \in \pi_{\DD_\pi}$. We consider two cases.

  (Case 1): $E_i = \{ i \}$. In this case, by construction, there is no arc either starting or ending at $i$ in $\DD_\pi$. This means that there is no $F \in \pi$ with $\# F > 1$ which contains $i$. Thus $\{i \} \in \pi$ should hold by (NCL1), that is, $E_i \in \pi$.

  (Case 2): $\# E_i > 1$. By construction, there is some arc $(i,j)$ in $\DD_\pi$, thus there is some $E \in \pi$ satisfying $i = \min E$ and $j \in E$. It suffices to show $E = E_i$.
  If $j' \in E$ with $j' \neq i$, then $(i,j') \in \DD_\pi$ holds by construction. Thus $j' \in E_i$ holds, and we obtain $E \subset E_i$.
  Conversely, suppose $j' \in E_i$ with $j' \neq i$. Then $(i,j') \in \DD_\pi$, so there is some $E' \in \pi$ with $i = \min E'$ and $j' \in E'$. Then $i \in E \cap E'$ satisfies $i = \min E = \min E'$, which implies $E = E'$ by (NCL3). Thus $j' \in E' = E$ holds. Therefore, we have $E_i = E \in \pi$.

  Hence we obtain $\pi = \pi_{\DD_\pi}$, which completes the proof.

  (2)
  We will show the equality (\ref{eq2}) by calculating the generating function using (\ref{eq1}). Let $a_n$ denote the number of mono-crossing admissible arc diagrams on $[n]$, and put $b_n:= \#\mbrick B_n$. We have that $a_n$ is equal to the $(n-1)$-th large Schr\"oder number by (1).

  Consider the following generating functions.
  \begin{align*}
    f(t) &:= \sum_{n=1}^\infty a_n t^n, \\
    g(t) &:= \sum_{n=1}^\infty b_n t^n.
  \end{align*}
  It is known that the following holds for the large Schr\"oder number (see e.g. \cite[Theorem 8.5.7]{bru}):
  \[
  f(t) = \frac{1-t-\sqrt{1-6t+t^2}}{2}
  \]
  To compute $g(t)$, we claim the following relation between $a_n$ and $b_n$.

  {\bf (Claim)}: \emph{The equality $b_n = a_n + \sum_{i=1}^n i\, a_i \,a_{n+1-i}$ holds}.

  \emph{Proof of (Claim)}:
  Let $\mathsf{MD}_n$ be the set of mono-crossing arc diagrams on $[n]$, so $\# \mathsf{MD}_n = b_n$. For $0 \leq i \leq n$, we define $\mathsf{MD}_n(i)$ as follows:
  $\mathsf{MD}_n (0)$ consists of $\DD \in \mathsf{MD}_n$ such that there is no arc in $\DD$ whose socle series contains $n$, and for $1 \leq i \leq n$, $\mathsf{MD}_n(i)$ consists of $\DD \in \mathsf{MD}_n$ such that there is an arc in $\DD$ whose socle series contains $n$, and the maximal length of such arcs is $i$. Here the length of the arc $(i,j)$ is defined to be $j-i \in [n]$. Then clearly we have the following decomposition, and  we will count the number of elements in $\mathsf{MD}_n(i)$.
  \[
  \mathsf{MD}_n = \bigsqcup_{i=0}^n \mathsf{MD}_n(i)
  \]

  For $i=0$, clearly elements in $\mathsf{MD}_n(0)$ are precisely mono-crossing \emph{admissible} arc diagrams on $[n]$. Thus $\#\mathsf{MD}_n(0) = a_n$ holds.

  Let $1 \leq i \leq n$. There are $i$ arcs whose socle series contain $n$ and whose lengths are $i$, that is, $(n-i+1,1), (n-i+2,2), \dots, (n,n+i)$. Elements in $\mathsf{MD}_n(i)$ contain precisely one such arc, since any two such arcs are strictly crossing.
  Fix one such arc $\al =(j,j+i)$, and we will count the number of elements in $\mathsf{MD}_n(i)$ which contain $\al$. Let $\DD$ be such an element. Then consider the restriction of $\DD$ to the part $\{j,j+1,\dots,j+i\}$, more precisely, consider the set of arcs whose socle series are partial sequences of that of $\al$.
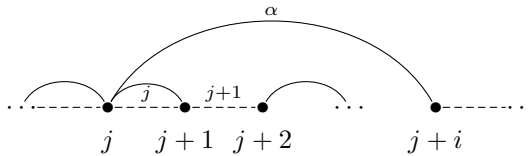
  \[
  \begin{tikzcd}[row sep=0mm, nodes={inner sep=0pt}]
    \cdots \ar[-,r, bend left = 60] \rar[-,dashed]
    & \bullet \ar[loop below,phantom, "j"] \ar[-,rrrr, bend left = 60, "\al"] \ar[-,r, bend left = 60] \rar[-,dashed, "j"]
    & \bullet \ar[loop below,phantom, "j+1"] \rar[-,dashed, "j+1"]
    & \bullet \ar[loop below,phantom, "j+2"]\ar[-,r, bend left = 60]
    & \cdots
    & \bullet \ar[loop below,phantom, "j+i"] \rar[-,dashed]
    & \cdots
  \end{tikzcd}
  \]
  By shifting $-(j-1)$, these arcs except $\al$ give a mono-crossing admissible arc diagram on $[i]$ (note that it is not $[i+1]$, since arcs cannot share the endpoint with $\al$). Conversely, any mono-crossing admissible arc diagram on $[i]$ can occur in this way by shifting $+(j-1)$.

  In a similar way, consider the set of arcs in $\DD$ whose socle series are disjoint from that of $\al$. By shifting $-(j+i-1)$, these arcs give a mono-crossing admissible arc diagram on $[n-i+1]$, and vice versa. Therefore, there are $i \cdot a_i \cdot a_{n-i+1}$ possible arc diagrams in $\mathsf{MD}_n(i)$. $\qedb$

  Now, by using (Claim), we obtain the following equality.
  \begin{align*}
    g(t) &= \sum_{n=1}^\infty b_n t^n \\
    &= (a_1 + 1 a_1  a_1) t + (a_2 + 1  a_1  a_2 + 2 a_2  a_1) t^2 +
    (a_3 + 1  a_1  a_3 + 2 a_2 a_2 + 3 a_3 a_1) t^3 + \cdots \\
    &= (a_1 t + a_2 t^2 + a_3 t^3 + \cdots) \cdot (1 + a_1 + 2 a_2 t + 3 a_3 t^2 + \cdots) \\
    &= f(t) \cdot (1+ \frac{df(t)}{dt}) \\
    &= \frac{1-t-\sqrt{1-6t+t^2}}{2} \cdot \frac{1}{2} \left(\frac{3-t}{\sqrt{1-6t+t^2}} + 1 \right) \\
    &= \frac{1}{2}\left( \frac{1+t}{\sqrt{1-6t+t^2}} -1 \right)
  \end{align*}
  Since this coincides with the generating function of \cite[A002003]{OEIS}, we are done.
\end{proof}

\begin{remark}
  In the paper \cite{eno:ICE}, we will compute the number of monobricks in $\mod A_n$ by using a completely different method. In fact, the number of subcategories in $\mod kQ$ closed under extensions, kernels and images (thus left Schur) is determined in \cite{eno:ICE} for a Dynkin quiver $Q$.
\end{remark}

\section{Examples of computations} \label{sec:ex}
\emph{In what follows, we fix an algebraically closed field $k$.}
For several finite-dimensional algebras $\Lambda$, we list all monobricks and left Schur subcategories, and we discuss the behavior of the maps $\WWW \colon \lSchur\Lambda \defl \wide\Lambda$ and $\FFF \colon \lSchur \Lambda \defl \torf\Lambda$ in terms of their counterparts $\mmax \colon \mbrick\Lambda \defl \sbrick\Lambda$ and $(\ov{-}) \colon \mbrick\Lambda \defl \ccmbrick \Lambda$.

\begin{example}
  Let $Q$ be the quiver $1 \ot 2 \ot 3$, then the AR quiver of $\mod kQ$ is given in Table \ref{tb:ar1}.
  By Theorem \ref{thm:nakcount}, we have $\#\mbrick kQ = 22$, the third Schr\"oder number. There are $1 + 6$ monobricks $\MM$ with $\#\MM \leq 1$, namely, the empty set, and a singleton for each indecomposable $kQ$-module.

  In Table \ref{123ex}, we list the remaining $15$ monobricks, together with their poset structure. For example, the notation $1<\substack{2\\1},3$ means that this poset consists of the disjoint union of two chains $1 < \substack{2\\1}$ and $3$.
  For each monobrick $\MM$, we write the corresponding left Schur subcategory $\Filt\MM$ in the AR quiver, where the black vertices are $\MM$, and the white vertices denote the remaining objects in $\Filt \MM$.
  If $\Filt\MM$ is not a wide subcategory, then we write the monobrick corresponding to $\WWW(\Filt\MM)$, which is equal to $\mmax \MM$ by Theorem \ref{thm:wideproj}.
  Similarly, if $\Filt\MM$ is not a torsion-free class, then we write the monobrick corresponding to $\FFF(\Filt\MM)$, which is equal to the cofinal closure $\ov{\MM}$ by Theorem \ref{thm:torfproj}.

  \begin{table}[htp]
    \caption{The Auslander-Reiten quiver of $\mod k[1\ot 2 \ot 3]$} \label{tb:ar1}
    \begin{tikzpicture}[scale=.6]
      \node (1) at (0,0)  {$1$};
      \node (21) at (1,1)  {$\substack{2\\1}$};
      \node (2) at (2,0) {$2$};
      \node (321) at (2,2) {$\substack{3\\2\\1}$};
      \node (32) at (3,1) {$\substack{3\\2}$};
      \node (3) at (4,0) {$3$};
      \draw[->] (1) -- (21);
      \draw[->] (21) -- (2);
      \draw[->] (21) -- (321);
      \draw[->] (321) -- (32);
      \draw[->] (2) -- (32);
      \draw[->] (32) -- (3);
    \end{tikzpicture}

    \caption{Monobricks $\MM$ over $k [1 \ot 2 \ot 3]$ with $\#\MM \geq 2$}
    \label{123ex}
    \begin{tabular}{C|c|c|C|c|C}
    \text{$\MM$ (as a poset)} & left Schur subcats & wide? & \mmax \MM & torsion-free? & \ov{\MM} \\ \hline \hline

    1 < \substack{2 \\ 1} &
    \begin{tikzpicture}
      [baseline={([yshift=-.5ex]current bounding box.center)}, scale=0.4,  every node/.style={scale=0.5}]
      \node (1) at (0,0) [black] {};
      \node (21) at (1,1) [black] {};
      \node (2) at (2,0) {};
      \node (321) at (2,2) {};
      \node (32) at (3,1) {};
      \node (3) at (4,0) {};
      \node at (0,-.2) {};
      \node at (0,2.2) {};
      \draw[->] (1) -- (21);
      \draw[->] (21) -- (2);
      \draw[->] (21) -- (321);
      \draw[->] (321) -- (32);
      \draw[->] (2) -- (32);
      \draw[->] (32) -- (3);
    \end{tikzpicture}
    & No & \substack{2 \\ 1} & Yes & \text{itself}
    \\ \hline

    1 < \substack{3 \\ 2 \\ 1} &
    \begin{tikzpicture}
      [baseline={([yshift=-.5ex]current bounding box.center)}, scale=0.4,  every node/.style={scale=0.5}]
      \node (1) at (0,0) [black]  {};
      \node (21) at (1,1)  {};
      \node (2) at (2,0) {};
      \node (321) at (2,2) [black] {};
      \node (32) at (3,1) {};
      \node (3) at (4,0) {};
      \node at (0,-.2) {};
      \node at (0,2.2) {};

      \draw[->] (1) -- (21);
      \draw[->] (21) -- (2);
      \draw[->] (21) -- (321);
      \draw[->] (321) -- (32);
      \draw[->] (2) -- (32);
      \draw[->] (32) -- (3);
    \end{tikzpicture}
    & No & \substack{3 \\ 2 \\ 1} & No & 1< \substack{2 \\ 1} < \substack{3\\2\\1}
    \\ \hline

    1, 2 &
    \begin{tikzpicture}
      [baseline={([yshift=-.5ex]current bounding box.center)}, scale=0.4,  every node/.style={scale=0.5}]
      \node (1) at (0,0) [black] {};
      \node (21) at (1,1) [white] {};
      \node (2) at (2,0) [black]{};
      \node (321) at (2,2)  {};
      \node (32) at (3,1) {};
      \node (3) at (4,0) {};
      \node at (0,-.2) {};
      \node at (0,2.2) {};

      \draw[->] (1) -- (21);
      \draw[->] (21) -- (2);
      \draw[->] (21) -- (321);
      \draw[->] (321) -- (32);
      \draw[->] (2) -- (32);
      \draw[->] (32) -- (3);
    \end{tikzpicture}
    & Yes & \text{itself} & Yes & \text{itself}
    \\ \hline

    1, \substack{3 \\ 2} &
    \begin{tikzpicture}
      [baseline={([yshift=-.5ex]current bounding box.center)}, scale=0.4,  every node/.style={scale=0.5}]
      \node (1) at (0,0) [black] {};
      \node (21) at (1,1) {};
      \node (2) at (2,0) {};
      \node (321) at (2,2) [white] {};
      \node (32) at (3,1) [black] {};
      \node (3) at (4,0) {};
      \node at (0,-.2) {};
      \node at (0,2.2) {};

      \draw[->] (1) -- (21);
      \draw[->] (21) -- (2);
      \draw[->] (21) -- (321);
      \draw[->] (321) -- (32);
      \draw[->] (2) -- (32);
      \draw[->] (32) -- (3);
    \end{tikzpicture}
    & Yes & \text{itself} & No & 1,2< \substack{3 \\ 2}
    \\ \hline

    1, 3 &
    \begin{tikzpicture}
      [baseline={([yshift=-.5ex]current bounding box.center)}, scale=0.4,  every node/.style={scale=0.5}]
      \node (1) at (0,0) [black] {};
      \node (21) at (1,1) {};
      \node (2) at (2,0) {};
      \node (321) at (2,2)  {};
      \node (32) at (3,1)  {};
      \node (3) at (4,0) [black] {};
      \node at (0,-.2) {};
      \node at (0,2.2) {};

      \draw[->] (1) -- (21);
      \draw[->] (21) -- (2);
      \draw[->] (21) -- (321);
      \draw[->] (321) -- (32);
      \draw[->] (2) -- (32);
      \draw[->] (32) -- (3);
    \end{tikzpicture}
    & Yes & \text{itself} & Yes & \text{itself}
    \\ \hline

    \substack{2\\1} < \substack{3 \\ 2 \\ 1} &
    \begin{tikzpicture}
      [baseline={([yshift=-.5ex]current bounding box.center)}, scale=0.4,  every node/.style={scale=0.5}]
      \node (1) at (0,0) {};
      \node (21) at (1,1) [black] {};
      \node (2) at (2,0) {};
      \node (321) at (2,2) [black] {};
      \node (32) at (3,1) {};
      \node (3) at (4,0) {};
      \node at (0,-.2) {};
      \node at (0,2.2) {};

      \draw[->] (1) -- (21);
      \draw[->] (21) -- (2);
      \draw[->] (21) -- (321);
      \draw[->] (321) -- (32);
      \draw[->] (2) -- (32);
      \draw[->] (32) -- (3);
    \end{tikzpicture}
    & No & \substack{3 \\ 2 \\ 1} & No & 1 < \substack{2 \\ 1}< \substack{3\\2\\1}
    \\ \hline

    \substack{2\\1} ,3 &
    \begin{tikzpicture}
      [baseline={([yshift=-.5ex]current bounding box.center)}, scale=0.4,  every node/.style={scale=0.5}]
      \node (1) at (0,0) {};
      \node (21) at (1,1) [black] {};
      \node (2) at (2,0) {};
      \node (321) at (2,2) [white] {};
      \node (32) at (3,1) {};
      \node (3) at (4,0) [black] {};
      \node at (0,-.2) {};
      \node at (0,2.2) {};

      \draw[->] (1) -- (21);
      \draw[->] (21) -- (2);
      \draw[->] (21) -- (321);
      \draw[->] (321) -- (32);
      \draw[->] (2) -- (32);
      \draw[->] (32) -- (3);
    \end{tikzpicture}
    & Yes & \text{itself} & No & 1 < \substack{2 \\ 1}, 3
    \\ \hline

    2,\substack{3\\2\\1} &
    \begin{tikzpicture}
      [baseline={([yshift=-.5ex]current bounding box.center)}, scale=0.4,  every node/.style={scale=0.5}]
      \node (1) at (0,0) {};
      \node (21) at (1,1)  {};
      \node (2) at (2,0) [black]{};
      \node (321) at (2,2) [black] {};
      \node (32) at (3,1) {};
      \node (3) at (4,0) {};
      \node at (0,-.2) {};
      \node at (0,2.2) {};

      \draw[->] (1) -- (21);
      \draw[->] (21) -- (2);
      \draw[->] (21) -- (321);
      \draw[->] (321) -- (32);
      \draw[->] (2) -- (32);
      \draw[->] (32) -- (3);
    \end{tikzpicture}
    & Yes & \text{itself}  & No & 1 <\substack{3\\2\\1}, 2
    \\ \hline

    2 < \substack{3\\2}&
    \begin{tikzpicture}
      [baseline={([yshift=-.5ex]current bounding box.center)}, scale=0.4,  every node/.style={scale=0.5}]
      \node (1) at (0,0) {};
      \node (21) at (1,1)  {};
      \node (2) at (2,0) [black]{};
      \node (321) at (2,2) {};
      \node (32) at (3,1) [black] {};
      \node (3) at (4,0) {};
      \node at (0,-.2) {};
      \node at (0,2.2) {};

      \draw[->] (1) -- (21);
      \draw[->] (21) -- (2);
      \draw[->] (21) -- (321);
      \draw[->] (321) -- (32);
      \draw[->] (2) -- (32);
      \draw[->] (32) -- (3);
    \end{tikzpicture}
    & No & \substack{3\\2} & Yes & \text{itself}
    \\ \hline

    2 ,3&
    \begin{tikzpicture}
      [baseline={([yshift=-.5ex]current bounding box.center)}, scale=0.4,  every node/.style={scale=0.5}]
      \node (1) at (0,0) {};
      \node (21) at (1,1)  {};
      \node (2) at (2,0) [black]{};
      \node (321) at (2,2) {};
      \node (32) at (3,1) [white] {};
      \node (3) at (4,0)[black] {};
      \node at (0,-.2) {};
      \node at (0,2.2) {};

      \draw[->] (1) -- (21);
      \draw[->] (21) -- (2);
      \draw[->] (21) -- (321);
      \draw[->] (321) -- (32);
      \draw[->] (2) -- (32);
      \draw[->] (32) -- (3);
    \end{tikzpicture}
    & Yes & \text{itself} & Yes & \text{itself}
    \\ \hline

    1 < \substack{2 \\ 1} < \substack{3\\2\\1} &
    \begin{tikzpicture}
      [baseline={([yshift=-.5ex]current bounding box.center)}, scale=0.4,  every node/.style={scale=0.5}]
      \node (1) at (0,0) [black] {};
      \node (21) at (1,1) [black] {};
      \node (2) at (2,0) {};
      \node (321) at (2,2) [black] {};
      \node (32) at (3,1) {};
      \node (3) at (4,0) {};
      \node at (0,-.2) {};
      \node at (0,2.2) {};
      \draw[->] (1) -- (21);
      \draw[->] (21) -- (2);
      \draw[->] (21) -- (321);
      \draw[->] (321) -- (32);
      \draw[->] (2) -- (32);
      \draw[->] (32) -- (3);
    \end{tikzpicture}
    & No & \substack{3\\2\\1} & Yes & \text{itself}
    \\ \hline

    1 < \substack{2 \\ 1},3 &
    \begin{tikzpicture}
      [baseline={([yshift=-.5ex]current bounding box.center)}, scale=0.4,  every node/.style={scale=0.5}]
      \node (1) at (0,0) [black] {};
      \node (21) at (1,1) [black] {};
      \node (2) at (2,0) {};
      \node (321) at (2,2) [white] {};
      \node (32) at (3,1) {};
      \node (3) at (4,0)[black] {};
      \node at (0,-.2) {};
      \node at (0,2.2) {};
      \draw[->] (1) -- (21);
      \draw[->] (21) -- (2);
      \draw[->] (21) -- (321);
      \draw[->] (321) -- (32);
      \draw[->] (2) -- (32);
      \draw[->] (32) -- (3);
    \end{tikzpicture}
    & No & \substack{2\\1},3 & Yes & \text{itself}
    \\ \hline

    1 < \substack{3 \\ 2 \\ 1}, 2 &
    \begin{tikzpicture}
      [baseline={([yshift=-.5ex]current bounding box.center)}, scale=0.4,  every node/.style={scale=0.5}]
      \node (1) at (0,0) [black]  {};
      \node (21) at (1,1) [white] {};
      \node (2) at (2,0) [black]{};
      \node (321) at (2,2) [black] {};
      \node (32) at (3,1) {};
      \node (3) at (4,0) {};
      \node at (0,-.2) {};
      \node at (0,2.2) {};

      \draw[->] (1) -- (21);
      \draw[->] (21) -- (2);
      \draw[->] (21) -- (321);
      \draw[->] (321) -- (32);
      \draw[->] (2) -- (32);
      \draw[->] (32) -- (3);
    \end{tikzpicture}
    & No & 2,\substack{3 \\ 2 \\ 1} & Yes & \text{itself}
    \\ \hline

    1, 2< \substack{3\\2} &
    \begin{tikzpicture}
      [baseline={([yshift=-.5ex]current bounding box.center)}, scale=0.4,  every node/.style={scale=0.5}]
      \node (1) at (0,0) [black] {};
      \node (21) at (1,1) [white] {};
      \node (2) at (2,0) [black]{};
      \node (321) at (2,2) [white]  {};
      \node (32) at (3,1) [black]{};
      \node (3) at (4,0) {};
      \node at (0,-.2) {};
      \node at (0,2.2) {};

      \draw[->] (1) -- (21);
      \draw[->] (21) -- (2);
      \draw[->] (21) -- (321);
      \draw[->] (321) -- (32);
      \draw[->] (2) -- (32);
      \draw[->] (32) -- (3);
    \end{tikzpicture}
    & No & 1,\substack{3\\2} & Yes & \text{itself}
    \\ \hline

    1, 2,3 &
    \begin{tikzpicture}
      [baseline={([yshift=-.5ex]current bounding box.center)}, scale=0.4,  every node/.style={scale=0.5}]
      \node (1) at (0,0) [black] {};
      \node (21) at (1,1) [white] {};
      \node (2) at (2,0) [black]{};
      \node (321) at (2,2) [white]  {};
      \node (32) at (3,1) [white]{};
      \node (3) at (4,0) [black]{};
      \node at (0,-.2) {};
      \node at (0,2.2) {};

      \draw[->] (1) -- (21);
      \draw[->] (21) -- (2);
      \draw[->] (21) -- (321);
      \draw[->] (321) -- (32);
      \draw[->] (2) -- (32);
      \draw[->] (32) -- (3);
    \end{tikzpicture}
    & Yes & \text{itself} & Yes & \text{itself}
    \\ \hline
    \end{tabular}
  \end{table}

  Now let us see some specific examples of computation of $\mmax \MM$ and $\ov{\MM}$. For a given monobrick $\MM$, it is easy to describe its poset structure (we have $L \leq M$ in $\MM$ whenever there is a non-zero map $L \to M$). Thus its maximal elements $\mmax\MM$ can be easily computed.
  For example, consider $\MM = \{ 1,\substack{2\\1}, 3\}$. Then since we have an injection $1 \hookrightarrow \substack{2\\1}$ and there are no other non-zero homomorphisms between two distinct objects in $\MM$, its poset structure is $1 < \substack{2\\1}, 3$, hence $\mmax \MM = \{\substack{2\\1}, 3\}$.

  The computation of $\ov{\MM}$ is a little bit more involved than $\mmax\MM$. Recall from Corollary \ref{cor:closuredesc} that $\ov{\MM}$ consists of all bricks $N$ which satisfy the following two conditions:
  \begin{enumerate}
    \item $N$ is a submodule of some $M \in \MM$.
    \item Every map $N \to M'$ with $M' \in \MM$ is either zero or an injection.
  \end{enumerate}
  Thus, to compute $\ov{\MM}$, first list all submodules of elements in $\MM$ which are bricks and not in $\MM$, then check whether the condition (2) above holds.
  For example, let $\MM = \{ 2, \substack{3\\2\\1} \}$. Then proper submodules which are bricks are exactly $\substack{2\\1}$ and $1$. However, there is a non-zero non-injection $\substack{2\\1} \defl 2$, thus we exclude $\substack{2\\1}$. In this way we obtain $\ov{\MM} = \MM \cup \{ 1 \}$.
\end{example}

Next consider the path algebra of an $A_3$ quiver with another orientation.
\begin{example}\label{ex:213ex}
  Let $Q$ be the quiver $1 \to 2 \ot 3$.
  There are $1 + 6$ monobricks $\MM$ with $\#\MM \leq 1$, namely, the empty set, and a singleton for each indecomposable $kQ$-module. It turns out that $\#\mbrick kQ = 26$. This means that \emph{the number of left Schur subcategories (or monobricks) depends on the orientation of the quiver}.

  In Table \ref{213ex}, we list the remaining $19$ monobricks and their maximal elements and cofinal closures. Wide subcategories are categories in which $\mmax \MM$ is \emph{itself}, and torsion-free classes are categories in which $\ov{\MM}$ is \emph{itself}.
  In this case, there are several examples which are not closed under direct summands, kernels or images. Subcategories $\EE$ with (*) are not closed under direct summands (hence not closed under images or kernels either), and in this case the white vertices in $\EE$ indicate indecomposables of $\add\EE$ which do not belong to $\EE$. The only one subcategory with (**) is closed under images, thus closed under direct summands, but is not closed under kernels.
  The remaining subcategories are all closed under kernels and images, and there are 22 such subcategories, the same number as the previous example.
  This is not a coincidence, as explained in the next remark.
  \begin{table}[htp]
    \caption{The Auslander-Reiten quiver of $\mod k[1 \to 2 \ot 3]$}
    \begin{tikzpicture}[scale = .8]
      \node (1) at (3,2) {$1$};
      \node (12) at (1,0) {$\substack{1 \\ 2}$};
      \node (123) at (2,1) {$\substack{1\,3 \\ 2}$};
      \node (2) at (0,1) {$2$};
      \node (23) at (1,2) {$\substack{ 3 \\ 2}$};
      \node (3) at (3,0) {$3$};

      \draw[->] (2) -- (23);
      \draw[->] (2) -- (12);
      \draw[->] (23) -- (123);
      \draw[->] (12) -- (123);
      \draw[->] (123) -- (1);
      \draw[->] (123) -- (3);
    \end{tikzpicture}

    \caption{Monobricks $\MM$ over $k [1 \to 2 \ot 3]$ with $\#\MM \geq 2$}
    \label{213ex}
    \begin{tabular}{C|C|C|C}
      \MM & \EE &  \mmax \MM &  \ov{\MM} \\ \hline \hline

      2 < \substack{1  \\ 2} &
      \begin{tikzpicture}
        [baseline={([yshift=-.5ex]current bounding box.center)}, scale=0.4,  every node/.style={scale=0.5}]
        \node (1) at (3,2) {};
        \node (12) at (1,0) [black] {};
        \node (123) at (2,1) {};
        \node (2) at (0,1) [black] {};
        \node (23) at (1,2) {};
        \node (3) at (3,0) {};

        \draw[->] (2) -- (23);
        \draw[->] (2) -- (12);
        \draw[->] (23) -- (123);
        \draw[->] (12) -- (123);
        \draw[->] (123) -- (1);
        \draw[->] (123) -- (3);

        \node at (0,-.2) {};
        \node at (0,2.2) {};
      \end{tikzpicture}
      &  \substack{1  \\ 2} &  \text{itself}
      \\ \hline

      2 < \substack{ 3 \\ 2} &
      \begin{tikzpicture}
        [baseline={([yshift=-.5ex]current bounding box.center)}, scale=0.4,  every node/.style={scale=0.5}]
        \node (1) at (3,2) {};
        \node (12) at (1,0) {};
        \node (123) at (2,1) {};
        \node (2) at (0,1) [black] {};
        \node (23) at (1,2) [black]{};
        \node (3) at (3,0) {};

        \draw[->] (2) -- (23);
        \draw[->] (2) -- (12);
        \draw[->] (23) -- (123);
        \draw[->] (12) -- (123);
        \draw[->] (123) -- (1);
        \draw[->] (123) -- (3);

        \node at (0,-.2) {};
        \node at (0,2.2) {};
      \end{tikzpicture}
      &  \substack{3 \\ 2} &  \text{itself}
      \\ \hline

      2 < \substack{1 \, 3 \\ 2} & (*)
      \begin{tikzpicture}
        [baseline={([yshift=-.5ex]current bounding box.center)}, scale=0.4,  every node/.style={scale=0.5}]
        \node (1) at (3,2) {};
        \node (12) at (1,0) [white] {};
        \node (123) at (2,1) [black] {};
        \node (2) at (0,1) [black] {};
        \node (23) at (1,2) [white] {};
        \node (3) at (3,0) {};

        \draw[->] (2) -- (23);
        \draw[->] (2) -- (12);
        \draw[->] (23) -- (123);
        \draw[->] (12) -- (123);
        \draw[->] (123) -- (1);
        \draw[->] (123) -- (3);

        \node at (0,-.2) {};
        \node at (0,2.2) {};
      \end{tikzpicture}
      &  \substack{1\,3 \\ 2} &
      \NN
      \\ \hline

      2,1 &
      \begin{tikzpicture}
        [baseline={([yshift=-.5ex]current bounding box.center)}, scale=0.4,  every node/.style={scale=0.5}]
        \node (1) at (3,2) [black]{};
        \node (12) at (1,0) [white]{};
        \node (123) at (2,1) {};
        \node (2) at (0,1) [black] {};
        \node (23) at (1,2) {};
        \node (3) at (3,0) {};

        \draw[->] (2) -- (23);
        \draw[->] (2) -- (12);
        \draw[->] (23) -- (123);
        \draw[->] (12) -- (123);
        \draw[->] (123) -- (1);
        \draw[->] (123) -- (3);

        \node at (0,-.2) {};
        \node at (0,2.2) {};
      \end{tikzpicture}
      &  \text{itself} &  \text{itself}
      \\ \hline

      2,3 &
      \begin{tikzpicture}
        [baseline={([yshift=-.5ex]current bounding box.center)}, scale=0.4,  every node/.style={scale=0.5}]
        \node (1) at (3,2) {};
        \node (12) at (1,0) {};
        \node (123) at (2,1) {};
        \node (2) at (0,1) [black] {};
        \node (23) at (1,2) [white] {};
        \node (3) at (3,0) [black] {};

        \draw[->] (2) -- (23);
        \draw[->] (2) -- (12);
        \draw[->] (23) -- (123);
        \draw[->] (12) -- (123);
        \draw[->] (123) -- (1);
        \draw[->] (123) -- (3);

        \node at (0,-.2) {};
        \node at (0,2.2) {};
      \end{tikzpicture}
      &  \text{itself} &  \text{itself}
      \\ \hline

      \substack{1\\2} , \substack{3 \\2} &
      \begin{tikzpicture}
        [baseline={([yshift=-.5ex]current bounding box.center)}, scale=0.4,  every node/.style={scale=0.5}]
        \node (1) at (3,2) {};
        \node (12) at (1,0) [black] {};
        \node (123) at (2,1)  {};
        \node (2) at (0,1) {};
        \node (23) at (1,2) [black] {};
        \node (3) at (3,0) {};

        \draw[->] (2) -- (23);
        \draw[->] (2) -- (12);
        \draw[->] (23) -- (123);
        \draw[->] (12) -- (123);
        \draw[->] (123) -- (1);
        \draw[->] (123) -- (3);

        \node at (0,-.2) {};
        \node at (0,2.2) {};
      \end{tikzpicture}
      &  \text{itself} &
      \substack{3\\2} > 2 < \substack{1 \\ 2}
      \\ \hline

      \substack{1\\2} < \substack{1\,3 \\2} &
      \begin{tikzpicture}
        [baseline={([yshift=-.5ex]current bounding box.center)}, scale=0.4,  every node/.style={scale=0.5}]
        \node (1) at (3,2) {};
        \node (12) at (1,0) [black] {};
        \node (123) at (2,1) [black] {};
        \node (2) at (0,1) {};
        \node (23) at (1,2)  {};
        \node (3) at (3,0) {};

        \draw[->] (2) -- (23);
        \draw[->] (2) -- (12);
        \draw[->] (23) -- (123);
        \draw[->] (12) -- (123);
        \draw[->] (123) -- (1);
        \draw[->] (123) -- (3);

        \node at (0,-.2) {};
        \node at (0,2.2) {};
      \end{tikzpicture}
      &  \substack{1\,3 \\2} &  \text{itself}
      \\ \hline

      \substack{ 3\\2} < \substack{1\,3 \\2} &
      \begin{tikzpicture}
        [baseline={([yshift=-.5ex]current bounding box.center)}, scale=0.4,  every node/.style={scale=0.5}]
        \node (1) at (3,2) {};
        \node (12) at (1,0)  {};
        \node (123) at (2,1) [black] {};
        \node (2) at (0,1) {};
        \node (23) at (1,2) [black] {};
        \node (3) at (3,0) {};

        \draw[->] (2) -- (23);
        \draw[->] (2) -- (12);
        \draw[->] (23) -- (123);
        \draw[->] (12) -- (123);
        \draw[->] (123) -- (1);
        \draw[->] (123) -- (3);

        \node at (0,-.2) {};
        \node at (0,2.2) {};
      \end{tikzpicture}
      &  \substack{1\,3 \\2} &  \text{itself}
      \\ \hline

      \substack{1\\2} , 3 &
      \begin{tikzpicture}
        [baseline={([yshift=-.5ex]current bounding box.center)}, scale=0.4,  every node/.style={scale=0.5}]
        \node (1) at (3,2) {};
        \node (12) at (1,0) [black] {};
        \node (123) at (2,1) [white] {};
        \node (2) at (0,1) {};
        \node (23) at (1,2)  {};
        \node (3) at (3,0) [black] {};

        \draw[->] (2) -- (23);
        \draw[->] (2) -- (12);
        \draw[->] (23) -- (123);
        \draw[->] (12) -- (123);
        \draw[->] (123) -- (1);
        \draw[->] (123) -- (3);

        \node at (0,-.2) {};
        \node at (0,2.2) {};
      \end{tikzpicture}
      &  \text{itself} &  2 < \substack{1\\2} , 3
      \\ \hline

      \substack{ 3\\2} , 1 &
      \begin{tikzpicture}
        [baseline={([yshift=-.5ex]current bounding box.center)}, scale=0.4,  every node/.style={scale=0.5}]
        \node (1) at (3,2) [black] {};
        \node (12) at (1,0)  {};
        \node (123) at (2,1) [white] {};
        \node (2) at (0,1) {};
        \node (23) at (1,2) [black] {};
        \node (3) at (3,0)  {};

        \draw[->] (2) -- (23);
        \draw[->] (2) -- (12);
        \draw[->] (23) -- (123);
        \draw[->] (12) -- (123);
        \draw[->] (123) -- (1);
        \draw[->] (123) -- (3);

        \node at (0,-.2) {};
        \node at (0,2.2) {};
      \end{tikzpicture}
      &  \text{itself} &  2 < \substack{3\\2} , 1
      \\ \hline
    \end{tabular}
    \quad
    \begin{tabular}{C|C|C|C}
      \MM & \EE &  \mmax \MM &  \ov{\MM} \\ \hline \hline

      1,3 &
      \begin{tikzpicture}
        [baseline={([yshift=-.5ex]current bounding box.center)}, scale=0.4,  every node/.style={scale=0.5}]
        \node (1) at (3,2) [black] {};
        \node (12) at (1,0)  {};
        \node (123) at (2,1) {};
        \node (2) at (0,1) {};
        \node (23) at (1,2)  {};
        \node (3) at (3,0) [black] {};

        \draw[->] (2) -- (23);
        \draw[->] (2) -- (12);
        \draw[->] (23) -- (123);
        \draw[->] (12) -- (123);
        \draw[->] (123) -- (1);
        \draw[->] (123) -- (3);

        \node at (0,-.2) {};
        \node at (0,2.2) {};
      \end{tikzpicture}
      &  \text{itself} &  \text{itself}
      \\ \hline

      \substack{3\\2} > 2 < \substack{1 \\ 2}
      &
      \begin{tikzpicture}
        [baseline={([yshift=-.5ex]current bounding box.center)}, scale=0.4,  every node/.style={scale=0.5}]
        \node (1) at (3,2) {};
        \node (12) at (1,0) [black] {};
        \node (123) at (2,1)  {};
        \node (2) at (0,1) [black] {};
        \node (23) at (1,2) [black] {};
        \node (3) at (3,0) {};

        \draw[->] (2) -- (23);
        \draw[->] (2) -- (12);
        \draw[->] (23) -- (123);
        \draw[->] (12) -- (123);
        \draw[->] (123) -- (1);
        \draw[->] (123) -- (3);

        \node at (0,-.2) {};
        \node at (0,2.2) {};
      \end{tikzpicture}
      &  \substack{1\\2},\substack{3\\2} &  \text{itself}
      \\ \hline

      2 < \substack{1  \\ 2} < \substack{1\,3\\2} & (*)
      \begin{tikzpicture}
        [baseline={([yshift=-.5ex]current bounding box.center)}, scale=0.4,  every node/.style={scale=0.5}]
        \node (1) at (3,2) {};
        \node (12) at (1,0) [black] {};
        \node (123) at (2,1) [black] {};
        \node (2) at (0,1) [black] {};
        \node (23) at (1,2) [white] {};
        \node (3) at (3,0) {};

        \draw[->] (2) -- (23);
        \draw[->] (2) -- (12);
        \draw[->] (23) -- (123);
        \draw[->] (12) -- (123);
        \draw[->] (123) -- (1);
        \draw[->] (123) -- (3);

        \node at (0,-.2) {};
        \node at (0,2.2) {};
      \end{tikzpicture}
      &  \substack{1\,3 \\ 2} &
      \NN
      \\ \hline

      2 < \substack{3 \\ 2} < \substack{1\,3\\2} & (*)
      \begin{tikzpicture}
        [baseline={([yshift=-.5ex]current bounding box.center)}, scale=0.4,  every node/.style={scale=0.5}]
        \node (1) at (3,2) {};
        \node (12) at (1,0) [white] {};
        \node (123) at (2,1) [black] {};
        \node (2) at (0,1) [black] {};
        \node (23) at (1,2) [black] {};
        \node (3) at (3,0) {};

        \draw[->] (2) -- (23);
        \draw[->] (2) -- (12);
        \draw[->] (23) -- (123);
        \draw[->] (12) -- (123);
        \draw[->] (123) -- (1);
        \draw[->] (123) -- (3);

        \node at (0,-.2) {};
        \node at (0,2.2) {};
      \end{tikzpicture}
      &  \substack{1\,3 \\ 2} &
      \NN
      \\ \hline

      2 < \substack{1  \\ 2},3 &
      \begin{tikzpicture}
        [baseline={([yshift=-.5ex]current bounding box.center)}, scale=0.4,  every node/.style={scale=0.5}]
        \node (1) at (3,2) {};
        \node (12) at (1,0) [black] {};
        \node (123) at (2,1) [white]{};
        \node (2) at (0,1) [black] {};
        \node (23) at (1,2) [white]{};
        \node (3) at (3,0) [black]{};

        \draw[->] (2) -- (23);
        \draw[->] (2) -- (12);
        \draw[->] (23) -- (123);
        \draw[->] (12) -- (123);
        \draw[->] (123) -- (1);
        \draw[->] (123) -- (3);

        \node at (0,-.2) {};
        \node at (0,2.2) {};
      \end{tikzpicture}
      &  \substack{1  \\ 2},3 &  \text{itself}
      \\ \hline

      2 < \substack{3 \\ 2},1 &
      \begin{tikzpicture}
        [baseline={([yshift=-.5ex]current bounding box.center)}, scale=0.4,  every node/.style={scale=0.5}]
        \node (1) at (3,2) [black]{};
        \node (12) at (1,0) [black] {};
        \node (123) at (2,1) [white]{};
        \node (2) at (0,1) [black] {};
        \node (23) at (1,2) [white]{};
        \node (3) at (3,0) {};

        \draw[->] (2) -- (23);
        \draw[->] (2) -- (12);
        \draw[->] (23) -- (123);
        \draw[->] (12) -- (123);
        \draw[->] (123) -- (1);
        \draw[->] (123) -- (3);

        \node at (0,-.2) {};
        \node at (0,2.2) {};
      \end{tikzpicture}
      &  \substack{3\\ 2},1 &  \text{itself}
      \\ \hline

      \substack{3 \\ 2} < \substack{1\,3 \\2} > \substack{1 \\ 2}
      & (**)
      \begin{tikzpicture}
        [baseline={([yshift=-.5ex]current bounding box.center)}, scale=0.4,  every node/.style={scale=0.5}]
        \node (1) at (3,2) {};
        \node (12) at (1,0) [black] {};
        \node (123) at (2,1) [black] {};
        \node (2) at (0,1) {};
        \node (23) at (1,2) [black] {};
        \node (3) at (3,0) {};

        \draw[->] (2) -- (23);
        \draw[->] (2) -- (12);
        \draw[->] (23) -- (123);
        \draw[->] (12) -- (123);
        \draw[->] (123) -- (1);
        \draw[->] (123) -- (3);

        \node at (0,-.2) {};
        \node at (0,2.2) {};
      \end{tikzpicture}
      &  \substack{1\,3 \\2} &
      \NN
      \\ \hline

      1,2,3 &
      \begin{tikzpicture}
        [baseline={([yshift=-.5ex]current bounding box.center)}, scale=0.4,  every node/.style={scale=0.5}]
        \node (1) at (3,2) [black]{};
        \node (12) at (1,0) [white] {};
        \node (123) at (2,1) [white]{};
        \node (2) at (0,1) [black] {};
        \node (23) at (1,2) [white]{};
        \node (3) at (3,0) [black]{};

        \draw[->] (2) -- (23);
        \draw[->] (2) -- (12);
        \draw[->] (23) -- (123);
        \draw[->] (12) -- (123);
        \draw[->] (123) -- (1);
        \draw[->] (123) -- (3);

        \node at (0,-.2) {};
        \node at (0,2.2) {};
      \end{tikzpicture}
      &  \text{itself} &  \text{itself}
      \\ \hline

      \NN= \adjustbox{scale=.7}{
        \begin{tikzcd}[column sep=0, row sep=0]
          & \substack{3 \\ 2} \ar[rd, phantom, "<",sloped] \\
          2 \ar[ru, phantom, "<",sloped]\ar[rd, phantom, "<",sloped]& & \substack{1\,3 \\2} \\
          & \substack{1 \\ 2} \ar[ru, phantom, "<",sloped]
        \end{tikzcd} }
      &
      \begin{tikzpicture}
        [baseline={([yshift=-.5ex]current bounding box.center)}, scale=0.4,  every node/.style={scale=0.5}]
        \node (1) at (3,2) {};
        \node (12) at (1,0) [black] {};
        \node (123) at (2,1) [black]{};
        \node (2) at (0,1) [black] {};
        \node (23) at (1,2) [black]{};
        \node (3) at (3,0) {};

        \draw[->] (2) -- (23);
        \draw[->] (2) -- (12);
        \draw[->] (23) -- (123);
        \draw[->] (12) -- (123);
        \draw[->] (123) -- (1);
        \draw[->] (123) -- (3);

        \node at (0,-.2) {};
        \node at (0,2.2) {};
      \end{tikzpicture}
      & \substack{1\,3\\2} &  \text{itself}
      \\ \hline
    \end{tabular}
  \end{table}
\end{example}

\begin{remark}
  In \cite{eno:ICE}, it is shown that the number of subcategories of $\mod kQ$ which are closed under extensions, kernels and images does not depend on the orientation of the underlying graph for a Dynkin quiver $Q$, although the number of monobricks does depend on the orientation as we have seen.
  In particular, if $Q$ is of type $A_n$, then the number of such subcategories is equal to the $n$-th large Schr\"oder number by Theorem \ref{thm:nakcount}.
\end{remark}

The next example is non-hereditary, which already appeared in the introduction.
\begin{example}\label{ex:2nak}
  Let $\Lambda$ be \emph{any} Nakayama algebra whose quiver is $1 \rightleftarrows 2$. Then there are four bricks in $\mod\Lambda$, namely, $\brick\Lambda = \{ 1,2, \substack{1\\2}, \substack{2\\1}\}$. By using this (and without any consideration of other modules), we obtain the list of monobricks in Table \ref{nakex}.
  \begin{table}[htp]
    \caption{Monobricks over cyclic Nakayama algebras with $2$ simples}
    \label{nakex}
    \begin{tabular}{C|C|C}
      \MM &   \mmax \MM &  \ov{\MM} \\ \hline \hline
      \varnothing & \text{itself} & \text{itself} \\ \hline
      1 & \text{itself} & \text{itself} \\ \hline
      2 & \text{itself} & \text{itself} \\ \hline
      \substack{1\\2} & \text{itself} & 2 < \substack{1\\2} \\ \hline
    \end{tabular}
    \quad
    \begin{tabular}{C|C|C}
      \MM &   \mmax \MM &  \ov{\MM} \\ \hline \hline
      \substack{2\\1} & \text{itself} & 1 < \substack{2\\1} \\ \hline
      1 < \substack{2\\1} & \substack{2\\1} & \text{itself} \\ \hline
      2 < \substack{1\\2} & \substack{1\\2} & \text{itself} \\ \hline
      1,2 & \text{itself} & \text{itself} \\ \hline
    \end{tabular}
  \end{table}
\end{example}

Finally, we consider a representation-infinite case.
\begin{example}\label{ex:kro}
  Let $Q$ be the 2-Kronecker quiver, that is, $Q =[ 1 \leftleftarrows 2 ]$. Then a complete classification of indecomposable $kQ$-modules is known, see e.g. \cite[Section VIII.7]{ARS}. By using this, we obtain the following three classes of bricks.
  \begin{enumerate}
    \item Indecomposable preprojective modules $\{P_1,P_2,P_3,\dots\}$.
    \item Regular simple modules $\{ R_\lambda\}_{\lambda \in \P^1(k)}$.
    \item Indecomposable preinjective modules $\{I_1,I_2,I_3,\dots \}$.
  \end{enumerate}
  Here $P_1 = P(1), P_2 = P(2), P_3 = \tau^- P_1, P_4 = \tau^- P_2 , P_5 = \tau^- P_3, \dots$ and $I_1 = I(2), I_2 = I(1), I_3 = \tau I_1, I_4 = \tau I_2 ,\dots$, where $P(i)$ (resp. $I(i)$) is the indecomposable projective (resp. injective) module corresponding to the vertex $i$, and $\tau$ is the Auslander-Reiten translation.

  To classify monobricks over $kQ$, we need to know the lists of pairs $(B_1,B_2)$ of bricks such that there is a non-zero non-injection from $B_1$ to $B_2$, and pairs such that there is an injection but no non-zero non-injection from $B_1$ to $B_2$.
  This is summarized in Figure \ref{fig:kro}, where $B_1 \rightsquigarrow B_2$ (resp. $B_1 \hookrightarrow B_2$) indicates that there is a non-zero non-injection $B_1 \to B_2$ (resp. an injection but no non-zero non-injection).

  \begin{figure}[htp]
    \centering
    \caption{Structures of bricks in $\mod k[1 \leftleftarrows 2]$}\label{fig:kro}
    \begin{tikzcd}
      P_1 \rar[hookrightarrow] \ar[rd,hookrightarrow] & P_2 \rar[hookrightarrow]\dar[rightsquigarrow] & P_3 \ar[dl,rightsquigarrow] \rar[hookrightarrow] & \cdots \ar[dll,rightsquigarrow] \\
      & \text{any } R_\lambda \ar[ld,rightsquigarrow]\dar[hookrightarrow]\ar[rd, hookrightarrow] \ar[rrd, hookrightarrow] \\
      I_1 & I_2 \lar[rightsquigarrow] & I_3 \lar[rightsquigarrow] & \cdots \lar[rightsquigarrow]
    \end{tikzcd}
  \end{figure}

  Any other pairs can be deduced from the composition of arrows in Figure \ref{fig:kro}.
  Since there are lots of monobricks, we only consider cofinally closed monobricks. This is enough for classifying monobricks since a set of bricks is a monobrick if and only if it is a subset of some cofinally closed monobrick by Corollary \ref{cor:mbrickchar}.

  The following is the list of all cofinally closed monobricks, or the list of simple objects in all torsion-free classes.
  \begin{enumerate}
    \item[(M0)] $\varnothing$, the empty set.
    \item[(M1)] $\{ P_1, P_2,\dots, P_i\}$ for some $i$.
    \item[(M2)] $\{ P_1,P_2,P_3,\dots\}$.
    \item[(M3)] $\{P_1\} \cup \{R_\lambda\}_{\lambda \in X}$ for any non-empty subset $X \subset \P^1(k)$.
    \item[(M4)] $\{P_1\} \cup \{R_\lambda\}_{\lambda \in \P^1(k)} \cup \{I_i\}$ for $i \geq 2$.
    \item[(M5)] $\{I_1\}$.
    \item[(M6)] $\{P_1,I_1\}$.
  \end{enumerate}
  In this list, finite monobricks are (M0), (M1), (M3) for a finite set $X$, (M5) and (M6). The poset structure is as follows.
  \begin{table}[htp]
    \begin{tabular}{C|C|C|C|C|C|C}
      (M0) & (M1) & (M2) & (M3) & (M4) & (M5) & (M6) \\ \hline \hline
      \varnothing & P_1 < P_2 < \cdots < P_i & P_1 < P_2 < \cdots
      &
      \begin{tikzpicture}[baseline={([yshift=-.5ex]current bounding box.center)}]
        \node at (0,0) (P1) {$P_1$};
        \node at (1,.5) (R1) {$R_\lambda$};
        \node at (1,-.5) (R2) {$R_{\lambda'}$};
        \draw[rounded corners] (1.5,.7) rectangle (.5,-.7);
        \node at (1,0) {$\vdots$};
        \node[scale=.8] at (1.7,-.7) {$X$};
        \draw[right hook->] (P1) -- (R1);
        \draw[right hook->] (P1) -- (R2);
      \end{tikzpicture}
      &
      \begin{tikzpicture}[baseline={([yshift=-.5ex]current bounding box.center)}]
        \node at (0,0) (P1) {$P_1$};
        \node at (1,.5) (R1) {$R_\lambda$};
        \node at (1,-.5) (R2) {$R_{\lambda'}$};
        \draw[rounded corners] (1.5,.7) rectangle (.5,-.7);
        \node at (1,0) {$\vdots$};
        \node[scale=.8] at (2,-.7) {$\P^1(k)$};
        \node at (2,0) (I) {$I_i$};
        \draw[right hook->] (P1) -- (R1);
        \draw[right hook->] (P1) -- (R2);
        \draw[right hook->] (R1) -- (I);
        \draw[right hook->] (R2) -- (I);
      \end{tikzpicture}
      & I_1 & P_1, I_1
    \end{tabular}
    \caption{The poset structure of each monobrick}
  \end{table}

  Using this, we can easily compute $\WWW(\FF)$ for each torsion-free class, since $\WWW(\FF)$ is equal to $\Filt (\mmax (\simp\FF))$ by Theorem \ref{thm:wideproj}. Moreover, $\mmax(\simp\FF)$ is nothing but the brick labels starting at $\FF$ (Remark \ref{rem:label}), so we can compute the brick labels (c.f. \cite[Example 3.6]{DIRRT}). This can be summarized as follows.

  \begin{table}[htp]
    \begin{tabular}{C|C|C|C|C|C}
      (M1) & (M0), (M2) & (M3) & (M4) & (M5) & (M6) \\ \hline \hline
      \{P_i\} & \varnothing
      & \{ R_\lambda\}_{\lambda \in X}
      & \{I_i\} & \{I_1\} & \{ P_1, I_1\}
    \end{tabular}
    \caption{The maximal elements of each monobrick, or all the semibricks}
  \end{table}

  Since $\mmax \colon \ccmbrick kQ \to \sbrick kQ$ is surjective by Proposition \ref{prop:msid}, this table can also be seen as a table of all semibricks.

  We remark that if $X$ consists of one element in (M3), then the monobrick is isomorphic to $P_1 < P_2$ as posets, although the former corresponds to a non-(functorially finite) torsion-free class but the latter to a functorially finite one.
\end{example}

\begin{ack}
  The author would like to thank his supervisor Osamu Iyama for helpful comments and support. He also thanks Arashi Sakai for pointing out some mistakes and typos. He would like to thank the anonymous referee for his/her careful reading and valuable suggestions.
  This work is supported by JSPS KAKENHI Grant Number JP21J00299.
\end{ack}

\end{document}